\documentclass{article}

\usepackage[usenames,dvipsnames,condensed]{xcolor}

\usepackage{amsthm}
\usepackage{amsfonts}
\usepackage[UKenglish]{babel}
\usepackage[usenames]{xcolor}
\usepackage{graphicx}
\usepackage{soul}
\usepackage{stfloats}
\usepackage{morefloats}
\usepackage{cite}
\usepackage{lscape}

\usepackage{amsmath,amstext,amsopn,amsfonts,eucal,amssymb}
\usepackage{graphicx,wrapfig,url}

\newcommand\N{{\mathbb N}}
\newcommand\Z{{\mathbb Z}}

\newcommand\R{{\mathbb R}}

\newtheorem{theorem}{Theorem}[section]
\newtheorem{corollary}[theorem]{Corollary}
\newtheorem{lemma}[theorem]{Lemma}
\newtheorem{proposition}[theorem]{Proposition}
\newtheorem{definition}[theorem]{Definition}

\newtheorem{example}[theorem]{Example}

\newtheorem{remark}[theorem]{Remark}

\newtheorem{conjecture}[theorem]{Conjecture}

\newtheorem{problem}[theorem]{Problem}

\newtheorem*{theorem1.2}{Theorem 1.2}
\newtheorem*{theorem1.3}{Theorem 1.3}
\newtheorem*{corollary1.4}{Corollary 1.4}

\newcommand{\gr}{{\rm gr}}
\newcommand{\GR}{{\cal GR}}

\begin{document}
%\baselineskip1.04  \baselineskip

\title{
   Genus  Ranges of %Graphs and Double-Occurrence Words  }
$4$-Regular Rigid Vertex Graphs}

\author{Dorothy Buck$^a$, Egor Dolzhenko$^b$, Natasha Jonoska$^c$, Masahico Saito$^d$, Karin Valencia$^e$}
\date{\today}
\maketitle

\begin{flushleft}
\small{
\noindent $(a)$ d.buck@imperial.ac.uk, Imperial College London, South Kensington Campus, Department of Mathematics, Office: 623, London SW7 2AZ, England, Telephone number: +44 (0)20 758 48570
\\\noindent $(b)$ egor.dolzhenko@gmail.com, USF, Department of Mathematics and Statistics, Tampa, Florida 33620, U.S.A.
\\\noindent $(c)$ jonoska@math.usf.edu, USF, Department of Mathematics and Statistics, Tampa, Florida 33620, U.S.A.
\\\noindent $(d)$ saito@usf.edu, USF, Department of Mathematics and Statistics,
Tampa, Florida, 33620
U.S.A.
\\\noindent $(e)$  karin.valencia06@imperial.ac.uk,
Imperial College London, South Kensington Campus, Department of Mathematics, Office: 640, London SW7 2AZ, England, Telephone number: +44 (0)20 758 58625
}

\end{flushleft}

\begin{abstract}
We introduce a notion of {\it genus range} as a set of values of genera  over all 
 surfaces  into which a graph is embedded cellularly, and we study the genus ranges of 
 a special family of
 four-regular graphs with rigid vertices  that has been used in modeling homologous DNA recombination.
  We show that the genus 
  ranges are 
  sets of 
  consecutive integers. 
  For any 
 positive integer  $n$, 
   there are graphs with $2n $ vertices that have genus range $\{m,m+1,\ldots,m'\}$ for all 
$0\le m<m'\le n$, and there are graphs with $2n-1$ vertices with genus range $\{m,m+1,\ldots,m'\}$ for all 
$0\le m<m' <n$ or $0<m<m'\le n$.  Further, we show that for every $n$ there is $k<n$ such that $\{h\}$ is a
genus range for  graphs with $2n-1$ and $2n$ vertices for all $h\le k$. It is also shown that 
 for every $n$,  
 there is a graph with $2n$ vertices with genus range 
$\{0,1,\ldots,n\}$, 
but there is no such a graph with $2n-1$ vertices.

\end{abstract}
%-Abstract

\section{Introduction}

Genus of a graph (also known as minimum genus) is a well known notion in topological graph theory,
 and has been studied for a variety of graphs (e.g.,
 \cite{MoharThomassenBook}).
It represents the minimum genus of a surface in which a graph can be embedded.  
In this paper we introduce the notion of {\it genus range} of a given graph as the set of all possible genera 
of surfaces in which the graph can be embedded cellularly.

Graphs with 1-valent and 4-valent rigid vertices, called assembly graphs,  have been used to model homologous DNA rearrangements,
in which 
each $4$-valent vertex represents a recombination site
\cite{Angeleska2007,Angeleska2009}. Simple assembly graphs are graphs that can be specified by 
double occurrence words, or unsigned Gauss codes, and are closely related to virtual knot diagrams \cite{Kauff}. 
Spatial 
embeddings of these graphs could be seen as a physical representations of the molecules at the moment of rearrangement, hence it is of interest to study 
certain combinatorial properties and 
embeddings of a given graph
that reflect questions from such biological processes. 
Genara of graphs seem to measure their spacial %% added: their  spacial 
 complexity and are  %% added: are
  related to paths in the graphs that model assembled DNA segments.
Thus in this paper we study the genus ranges of assembly graphs. Two main questions are considered: 
\begin{problem}\label{MainProblem} { \   } % moves (a) in the next line

{\rm \noindent
(a) Characterize the sets of integers that appear as  genus ranges of assembly graphs with $n$ $4$-valent vertices  for each positive integer $n$.
 
 \noindent
 (b) Characterize the  assembly graphs with a given set of genus range.
}
\end{problem}
  With this paper we provide some answers to these questions. In particular, we show that 
a genus range of an assembly graph is always a set of consecutive integers (Lemma~\ref{conseclem}). 
Further, 
we show (Theorem~\ref{mainthm}) that every set $\{m, m+1, \ldots, m'\}$ for $0\le m<m'\le n$ appears as a genus range of some 
simple assembly graph with $2n$ vertices. The same family, without the set $\{0,1,\ldots,n\}$, appears as a genus range of some simple assembly graph with $2n-1$ vertices. We also observe that no simple assembly graph with $2n-1$ vertices has genus range $\{0,\ldots, n\}$ (Lemma~\ref{fullrangelem}) nor 
genus range  $\{n\}$ 
(Lemma~\ref{notnnlem}). 
We characterize the simple assembly graphs with genus range $\{0\}$.

 The paper is  organized  as follows.    In Section  \ref{PreliminaryDefs} we give the  definitions of assembly graphs, genus,  and ribbon graphs (thickened graphs) and preliminary observations.
  Boundary components of  ribbon graphs are closely related to 
  genera, so this section introduces  the basic techniques used in our results based on estimating the number of boundary components of ribbon graphs. 
  Section  \ref{PreliminaryDefs} 
  also contains results of computer calculations 
  and  
  a histogram of genus range distributions for graphs with 7 and 8 vertices. The algorithm in the computer program is 
   based on 
   the procedure described in \cite{Carter}. 
   A partial order on the sets of genus ranges
   is also introduced. 
   In Section~\ref{PropOfGenus}   some % the 
   properties of genus ranges
    are listed.. 
    In particular, we show that 
   no graph with $2n-1$ vertices can have genus range $\{0,\ldots,n\}$ nor $\{n\}$.
  In Section~\ref{Realization} 
  families of graphs that achieve certain 
   sets of genus ranges
   are constructed. 
   We characterize the assembly graphs with genus range $\{0\}$, and  
   give a family of graphs that 
  has genus range $\{n\}$. Further we provide a family of graphs with $2n$ vertices that have 
  genus range $\{0,1,\ldots,n\}$. 
  In Section~\ref{GenusRangeTC} we find the genus range of   
  {\it tangled cords}, a special subfamily of assembly graphs.
  This family achieves the 
  maximum genus range according to the  
  partial order for graphs with odd number of vertices. 
 The summary of results characterizing  
 genus ranges is given by 
Theorem~\ref{mainthm} in Section~\ref{Characterize}. We end the paper with some concluding remarks.

%%%%%%%%%%%%%%%%%%%%%%%%%%%%%%%%%%
\section{Terminology and Preliminaries}\label{PreliminaryDefs}
%%%%%%%%%%%%%%%%%%%%%%%%%%%%%%%%%%

In this section, definitions of the concepts used in this paper are recalled, notations are established, and 
their basic properties are listed. 

%%%%%%%%%%%%%%%%%%%%
\subsection{Double Occurrence Words and Assembly Graphs}
%%%%%%%%%%%%%%%%%%%%

A \textit{graph}
   is  a pair $(V, E)$ consisting of the set $V$ of vertices and the set $E$ of edges.
The endpoints of an edge are either a pair of vertices or a single vertex.
In the latter case, the edge is called a {\it loop}. 
The {\it degree} of a vertex $v$ is the number of edges incident to $v$ (each
loop is counted twice). 

A {\it $4$-valent rigid vertex} is a vertex of 
degree $4$
for which  a cyclic order of edges is specified.  
   For a $4$-valent rigid vertex $v$, if its incident edges appear in order
   $(e_1,e_2,e_3,e_4)$,  
   the cyclic orders equivalent to this order are 
   $(e_2,e_3,e_4,e_1)$, $(e_3,e_4,e_1,e_2)$, $(e_4,e_1,e_2,e_3)$, $(e_4,e_3,e_2,e_1)$, $(e_3,e_2,e_1,e_4)$,
   $(e_2,e_1,e_4,e_3)$,  and $(e_1,e_4,e_3,e_2)$.
 For the ordered edges   $(e_1,e_2,e_3,e_4)$,
we say that  $e_2$ and $e_4$ are  {\it neighbors} of $e_1$ and $e_3$ {\it with respect to $v$}
\cite{Angeleska2009}.

An {\it assembly graph}~\cite{Angeleska2009}
$\Gamma$ 
 is a finite connected graph where all vertices are 
 $4$-valent rigid vertices or vertices of degree $1$.
 A vertex of degree  
 $1$ is called an {\it endpoint}. 
In this paper, we focus  on assembly graphs with $4$-valent rigid vertices  only  
(without endpoints). 
Such graphs are studied in knot theory, and their spatial embeddings are also called singular knots and links
($\!\!$\cite{Dror}, for example). 

% Note that the definition of assembly graph implies that the number of endpoints is always even.
The number of $4$-valent vertices in
  $\Gamma$ is called the {\it size} of $\Gamma$ and  is denoted by $|\Gamma|$.
   An assembly graph is called {\it  trivial}
   if $|\Gamma|=0$. Two assembly graphs are {\it isomorphic} if they are isomorphic as 
   graphs and the graph isomorphism preserves the cyclic order of the edges incident to a vertex.

A path in  an assembly graph is called a {\it transverse path}, or simply  a {\it transversal}, if   consecutive edges of the path are 
never neighbors with respect to their common incident vertex. 
For an assembly graph without endpoints, a transversal is considered as the image 
of a circle in the graph, where the circle goes through every vertex ``straight''. 
An assembly graph that has an Eulerian (visiting all edges) 
 transversal is  called a {\it simple assembly graph}. 
 For a simple assembly graph, if a transversal is oriented, the graph is called {\it oriented}. 
  We note that in a simple assembly graph, if a vertex $v$ is an endpoint of
 a loop $e$, then $e$ must be a neighbor of itself.

\medskip   %%
\noindent %%%
{\bf Convention:} In the rest of the paper 
unless otherwise stated,  all graphs
are simple assembly graphs without endpoints. 
\medskip  %%

Let $A$ be an alphabet set. A {\it double-occurrence word}  (DOW) over $A$ is a word which contains
each symbol of $A$ exactly $0$ or $2$ times. 
DOWs  are %A DOW is 
  also called (unsigned) Gauss codes in knot theory (see, for example, \cite{Kauff}). 
The {\it reverse} of a word $w=a_1 \cdots a_k$ is $w^R=a_k \cdots a_1$. 
 Two DOWs 
 are called {\it equivalent}
  if one is obtained from the other by a sequence of the following 
  three operations: (1)  (bijective) renaming the symbols,  (2) a cyclic 
  permutation,   (3) taking the reverse. 
   For example,   $w=123231$ is equivalent to its reverse 
   $w^R=132321$ and
  $w'=213132$ is 
    $w^R$ after renaming $1$ with $2$ and $2$ with $1$.    Therefore, all these words are equivalent.
     If for a  DOW  %% added: for
 $w$ there is an equivalent DOW $w'$ such that  % can be written as a product 
 $w'=uv$,  where $u$ and $v$  are %% of 
 two non-empty DOWs 
  % $u, v$, where $w'$ is equivalent to $w$, 
then $w$ is called {\em
reducible};
otherwise, it is called {\it irreducible}.

    Double-occurrence words are related to assembly graphs as follows. 
Let $\Gamma$ be an oriented simple 
 assembly graph.  
Let  the set of $4$-valent vertices of $\Gamma$ be $V=\{v_1, \ldots, v_n\}$
 where $n=|\Gamma|$. 
 Pick and fix a base point on the graph and an orientation of a transversal. Starting from the base point, 
 travel along the transversal, and 
 write down the sequence of  vertices in 
 the order they  are encountered along  
 the 
  transversal. This  gives rise to   
 a DOW 
 over alphabet $V$. Conversely, for a given DOW
 containing $n$ letters,  %% added
 an assembly graph is constructed 
 with $n$ vertices (labeled with the letters)  %% added
 by tracing the labeled vertices
in  the order of their appearances in the DOW. %% added
It is known that
equivalence classes of DOWs  are in one-to-one correspondence with
isomorphism classes of assembly graphs \cite{Angeleska2009}. 
  In particular if $w$ is a DOW,
 we write $\Gamma_w$ for a simple assembly graph corresponding to $w$.

%%%%%%%%%%%%%%%%%%%%%%%%%%%%%
\subsection{Genus of  Simple Assembly Graphs}
%%%%%%%%%%%%%%%%%%%%%%%%%%%%%%%

 If a graph admits an embedding on the plane, it is called {\it planar}.
An embedding of a graph in a surface is called {\it cellular} if each component of the complement of the graph in the surface is an open disk.
 For a graph $G$, the {\it minimum orientable genus of $G$}, denoted $g_{min}(G)$, is the smallest non-negative integer $g$ such that $G$ admits an 
embedding in a closed (compact, without  boundary) orientable surface $F$ of genus $g$. 
We recall 
 the following fact 
 (Proposition 3.4.1 in \cite{MoharThomassenBook}): 
Every embedding of $G$ into a minimum genus surface is 
cellular.
 The {\it maximum orientable genus of $G$}, denoted $g_{max}(G)$,
 is the largest non-negative integer $g$ such that $G$ admits a cellular embedding in a closed  orientable surface $F$ of genus $g$. 
   In this work we are concerned only with embeddings in orientable surfaces. When an assembly graph is embedded in an orientable surface, we assume that the cyclic order of edges at each vertex
agrees with the embedding. In particular, neighboring edges at a vertex belong to the boundary of a common complementary region. 

\begin{definition}
{\rm 
The \textit{genus range} $\gr(\Gamma)$ of an assembly graph $\Gamma$ 
 is the set of values of genera over all 
 surfaces $F$ into which 
  $\Gamma$ 
is embedded cellularly. 
We denote the family of all genus ranges of assembly graphs with $n$ vertices 
by ${\cal GR}_n$.
}
\end{definition}

One way to obtain a cellular embedding of 
an assembly graph $\Gamma$
in a compact orientable surface is by constructing a surface by connecting bands (ribbons) along the graph. 
This construction is called a {\it ribbon construction}, 
and is used in estimating the number of DNA strands in DNA molecules representing graph structures\cite{Jonoska2002,Jonoska2009}. 
The construction is 
outlined below.

\begin{figure}[htb]
    \begin{center}
   \includegraphics[width=3in]{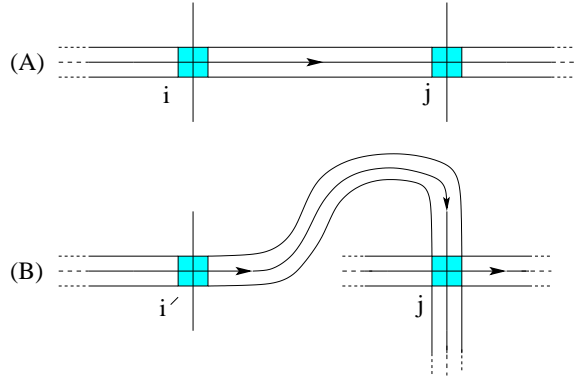}\\
    \caption{Ribbon construction }\label{vert}
    \end{center}
\end{figure}

Let $\Gamma$ be an assembly graph with $n$ vertices labeled from 
$1$ to $n$, and let $w$ be the DOW 
that represents $\Gamma$ with respect to a transversal.
 To each vertex $i$, $i=1 , \ldots, n$, 
we associate  a square with coordinate axes coincide with the edges incident to the vertex
as depicted in Figure  \ref{vert}. 
Each side of the square corresponds to an edge incident to the vertex such that neighboring sides of the square 
that   %% added
correspond to neighboring edges. 
For an edge $e$ with the end points $i$ and $j$, we connect the sides of the squares at $i$ and $j$ corresponding to $e$ 
by a band. 
The bands are attached in such a way that the resulting surface is orientable.
In  Figure  \ref{vert}(A), the connection by a band is described when the vertex $j$ 
immediately follows $i$ in $w$, where   %$i$ and 
$j$ is  %are 
  the first occurrence in $w$. 
In Figure  \ref{vert}(B), one possibility of connecting  a band  to the vertex $j$ 
at its second occurrence is shown. 
By continuing the band attachment along a transversal, one obtains a compact surface 
with boundary.
The resulting surface is called a  surface {\it obtained from ribbons}, {\it obtained by the ribbon construction},
or simply a {\it ribbon graph}.
This notion has been studied in literature for general graphs (see, for example, \cite{Sergei}).  
By capping the boundaries by disks, one obtains a cellular embedding of $\Gamma$ in an orientable, 
closed (without boundary) 
surface.    %% swap sentences For... and Note... 
For a given cellular embedding of an assembly graph $\Gamma$, its neighborhood is regarded as a surface obtained
from ribbons as described above. 
Note that there are two choices in connecting a band to a vertex $j$
at its second occurrence, either from the top as in Figure  \ref{vert}(B), or from the 
bottom. 
Hence 
$2^n$ ribbon graphs 
can be constructed that correspond to all  
cellular embeddings of a given assembly graph. 
Therefore, the genus range of a given assembly graph can be computed by 
finding genera of all surfaces constructed from ribbons. 
In \cite{CE1,CE2}, the two possibility of connecting bands are represented by signs ($\pm$), and signed Gauss codes were
used to specify the two choices.

%%%%%%
\subsection{Computer Calculations}
%%%%%%%

In this section we summarize our computer calculations and observations. 
The calculations help understand the structure of genus ranges and formulate conjectures.
Calculations are based on a description of boundary curves of ribbon graphs in \cite{Carter}.

\begin{remark}
{\rm
Computer calculations show that the sets of all possible genus ranges 
of $n$ letters for $n=2, \ldots, 7$ are as follows.
\begin{eqnarray*}
 & & {\cal GR}_n : \\
n=2 & &  \{ 0 \}, \{ 1 \}.  \\  
n=3 & & \{0\}, \{1\}, \{0,1\}, \{1,2\}.  \\
n=4 & & \{0\}, \{1\}, \{0,1\}, \{1,2\}, \{0,1,2\}.   \\
n=5 & &  \{0\}, \{1\}, \{2\}, \{0,1\}, \{1,2\}, \{2,3\}, \{0,1,2\}, \{1,2,3\}.  \\
n=6 & &   \{0\}, \{1\}, \{2\}, \{3\}, \{0,1\}, \{1,2\}, \{2,3\}, \{0,1,2\}, \{1,2,3\}, \{0,1,2,3\} .\\
n=7 & &   \{0\}, \{1\}, \{2\}, \{3\}, \{0,1\}, \{1,2\}, \{2,3\}, \{3,4\}, \\
 & &  \{0,1,2\}, \{1,2,3\}, \{2,3,4\}, \{0,1,2,3\}, \{1,2,3,4\}. \\
\end{eqnarray*}
For $n=8$, only  the set $\{0,1,2,3,4\}$ appears in addition to those for $n=7$. 
}
\end{remark}

\begin{remark}[Highest singleton genus ranges] \label{highest-single-rem}
{\rm 
Computer calculations show the following.
\begin{enumerate}
\setlength{\itemsep}{-3pt}
\item
Among all assembly graphs of $2$ vertices (there are only 
two,   %  added
$1122$ and $1212$), there is a unique 
graph, $1212$, that has the genus range $\{ 1 \}$. 

\item
Among all assembly graphs of $3$ vertices (there are $5$), there is a unique 
graph, $121233$, that has the genus range $\{ 1 \}$. 

\item
There is no assembly graph with $4$ vertices  that has the genus range $\{ 2 \}$. 

\item
Among all assembly graphs of $5$ vertices, there is a unique 
graph, $1234342515$, that has the genus range $\{ 2 \}$. 

\item
Among all assembly graphs of $6$ vertices, there is a unique 
graph, $123245153646$, that has the genus range $\{ 3 \}$. 

\item
\begin{sloppypar}
Among all assembly graphs of $7$ vertices, there are two
graphs, $12345416365277$ and $12324515364677$,
that 
have  % has 
  the genus range $\{ 3 \}$. 
  %%These are the above $6$ letter graph plus $77$.
The subgraphs obtained from these by deleting $77$ are equivalent (cyclic permutation and renaming) to the 
unique $6$ vertex  %letter 
 graph in the preceding case.

\end{sloppypar}

\item
There is no assembly graph with $8$ vertices  that has the genus range $\{ 4 \}$. 

\end{enumerate}
}
\end{remark}

\begin{remark}[Full genus ranges]\label{rem-full-range}
{\rm 
Computer calculations show the following.
\begin{enumerate}
\setlength{\itemsep}{-3pt}
\item
Among all assembly graphs of $4$ vertices, there is a unique 
graph, $12314324$, that has the genus range $\{ 0, 1, 2 \}$. 

\item
Among all assembly graphs of $6$ vertices, there is a unique 
graph, $123451256346$, that has the genus range $\{ 0, 1, 2, 3 \}$. 

\item
There are $13$ words out of $65346$ words of $8$ letters 
whose %%with 
corresponding 
assembly graphs have  %% added
 genus range $\{0,1,2,3,4\}$. 

\end{enumerate}
}
\end{remark}

\begin{remark}[Missing ranges]
{\rm
Computer calculations show the following.
\begin{enumerate}
\setlength{\itemsep}{-3pt}
\item
Among all assembly graphs of $2$ vertices (there are only 
two,  %%  added
$1122$ and $1212$), there is a unique 
graph, $1122$, that has the genus range $\{ 0 \}$, but no graph has the range $\{0,1\}$.
The situation is different from $4$ and $6$ vertices.

\item
Among all assembly graphs of $3$ vertices, there is a unique 
graph, $123123$, that has the genus range $\{ 0, 1\}$. 
There is a graph $121323$ with the range $\{1,2\}$, but  no graph has the range $\{0,1,2\}$.

\item
Among all assembly graphs of $5$ vertices, the graphs corresponding to the following 
words have the genus range $\{ 0, 1, 2\}$: 
 $1234214355$, $1231432455$,
 $1234215345$.
 Note that except the last one, they are obtained from a smaller graph by ``adding  a loop'' $55$ 
 (see Definition~\ref{loopnestdef} below for  a formal definition).
 There are graphs  with the range $\{2,3\}$ or $\{1,2, 3\}$, but  no graph has the range $\{0,1,2,3\}$.
 
\item
\begin{sloppypar}
Among all assembly graphs of $7$ vertices, graphs corresponding to the following
words have the genus range $\{ 0, 1, 2, 3\}$ and are not obtained from smaller graphs by ``addition of a loop'':  
$12345623417567$, 
$12345641237567$,
 $12345621637547$,
 $12345163267547$,
 $12345621657347$,
$12314564567327$, 
 $12314563267457$, 
 $12345645617327$.
 There is no word with the range $\{0,1,2,3,4\}$. 
 \end{sloppypar}
 
\end{enumerate}
}
\end{remark}

To put these calculations 
into %in a 
  perspective, we define the notation
 $[a, b]$ for the set $\{ a, \ldots, b\}$ for integers $0\leq a \leq b$.
We define the {\it consecutive power set} of $\{0, 1, \ldots, n\}$ for a positive integer $n$,
denoted by ${\cal CP}(n)$, to be the set of all consecutive positive integers:
${\cal CP}(n) = \{ \, [ a, b ] \, | \, 0 \leq a \leq b \leq n \}$. 
We define the following linear order on consecutive power sets.

\begin{definition}
{\rm
For 
$[a,b] , [c,d]  \in {\cal CP}(n)$, 
we say $[a,b] <   [c,d]$
if one of the following conditions holds:
(1) $b<d$, or
(2) $b=d$ and $a<c$. 
Define the partial order $\leq$ induced from this strict partial order $<$. 

}
\end{definition}

Note that from the definition it follows that 
the  %this 
 partial order $\leq$ on ${\cal CP}(n)$ 
 is  a linear order for any positive integer $n$.

\begin{figure}[htb]
    \begin{center}
   \includegraphics[width=2in]{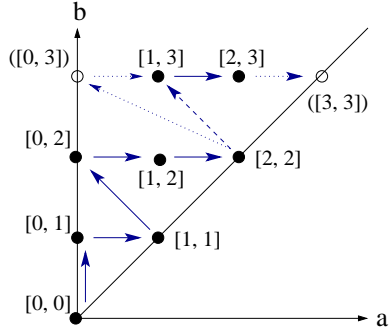}\\
    \caption{ Linear order on $\GR_5$ }\label{poset}
    \end{center}
\end{figure}

\begin{example}
{\rm
The Hasse diagrams of the  order on 
the consecutive power sets ${\cal CP}(3)$  and the restriction on $\GR_5$  are depicted in Figure  \ref{poset}.
The initial point of an arrow is the 
immediate predecessor of the terminal point. 
The sets $[0,3]=\{0,1,2,3\}$ 
and $[3,3]=\{3\}$ are not in $\GR_5$, and are enclosed in parentheses in  the figure. 
The order relation $[2,2]\leq [1,3]$ in $\GR_5$ is denoted by a dotted line,
which is immediate predecessor/successor  in $\GR_5$ but not in  ${\cal CP}( 3)$.

}
\end{example}

\begin{figure}
\begin{center}
  \includegraphics[scale=0.75]{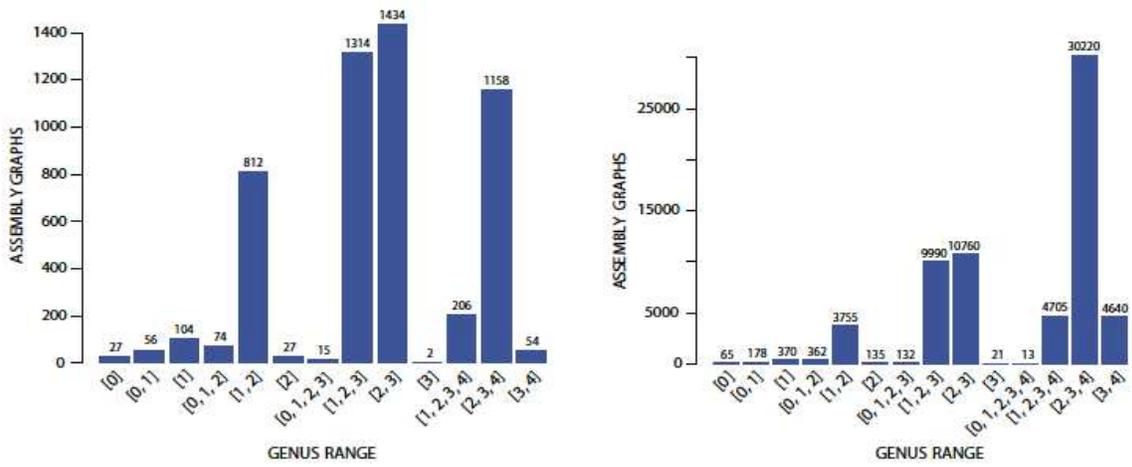} 
\end{center}
\vspace{-0.5cm}
\caption{ The number of assembly graphs of size 7 (left) and 8 (right) with a given genus range. }
\label{fig:genus_ranges_7_8}
\end{figure}

The genus ranges $\GR_7$ and $\GR_8$, as well as the number of graphs with each genus range
arranged in the linear order  are depicted in 
Figure   \ref{fig:genus_ranges_7_8}.

%%%%%%%%%%%%%%%%%%%%%%%%
\subsection{Computing the Genus Range of 
an Assembly  %a 
 Graph} 
%%%%%%%%%%%%%%%%%%%%%%%%

First we recall the well-known Euler characteristic formula, establishing the relation between 
the genus and the number of boundary components. 
The Euler characteristic $\chi(F)$ of a compact orientable surface
$F$ of genus $g(F)$ and the number of boundary components $b(F)$ are
related by $\chi(F)=2-2 g(F) - b(F)$. 
As a complex, $\Gamma$ is homotopic to a $1$-complex with $n$ vertices and $2n$
edges, and $F$ is homotopic to such a $1$-complex.
Hence $\chi(F)=n - 2n=-n$.
Thus we obtain the following well known formula, which we state as a lemma, as 
we will use it often in this paper.

\begin{lemma}\label{Eulerlem}
Let $F$ be a 
 surface for an assembly graph $\Gamma$ obtained by the ribbon construction. 
Let $g(F)$ be the genus, $b(F)$ be the number of boundary components of $F$, 
and $n$ be the number of vertices of $\Gamma$.
Then we have 
$g(F) =(1/2)(n-b(F) + 2 )$. 
\end{lemma}

Thus we can compute the genus range from  
the set of  the numbers of  boundary components of each ribbon graph,
$\{ \, b(F) \, | \, F\, ${\rm is a ribbon graph of}\,\, $ G\, \}$. 
Note that $n$ and $b(F)$ have the same parity, as genera are integers.

\begin{figure}[htb]
    \begin{center}
   \includegraphics[width=3.5in]{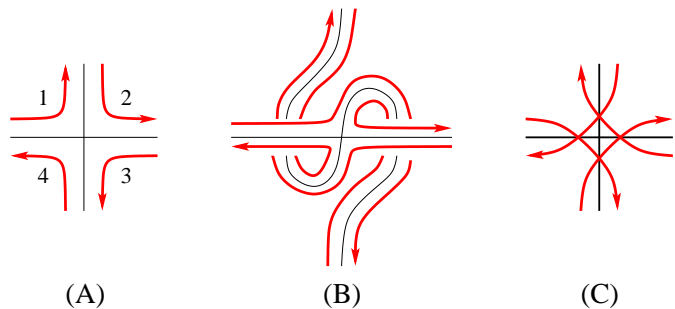}\\
    \caption{Changing the connection at a vertex }\label{flip}
    \end{center}
\end{figure}

Next we compute  the number of  boundary components of ribbon graphs 
for a given assembly graph.
In Figure  \ref{flip}(A), the boundary curves of a ribbon graph 
of an assembly graph, near a vertex, are indicated.
The arrows of these boundary curves indicate orientations of the boundary components induced from 
a chosen orientation of the ribbon graph. 
If in % In 
Figure  \ref{vert}(B), the  direction of entering 
the vertex has been changed from top to bottom (or vise versa), 
the ribbon graph changes.  %% added
This change 
in %s 
 the ribbon graph 
 is  %as 
 illustrated 
 in Figure~\ref{flip}(B).  %% added
  Note that the new
ribbon graph is orientable, as indicated by the arrows on the new boundary components.
We use a schematic image in Figure  \ref{flip}(C) to indicate the changes of connections of the boundary components illustrated in Figure ~\ref{flip}(B).  %\ref{vert}(B).
We call this operation a {\it connection change}. 
Thus starting from one ribbon graph for a given assembly graph $\Gamma$, 
one obtains its genus range by computing the number of boundary components for the surfaces obtained by switching 
connections at every vertex ($2^n$ possibilities for a graph with $|\Gamma|=n$ for a positive integer $n$).

\begin{figure}[htb]
    \begin{center}
   \includegraphics[width=3in]{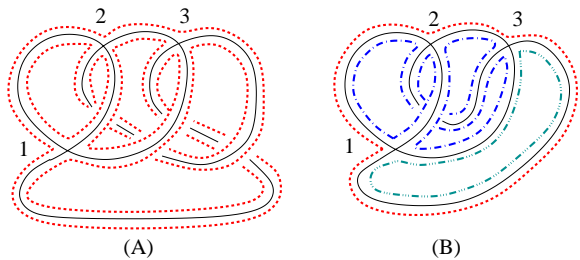}\\
    \caption{Boundary curves of  ribbon neighbourhoods the tangled cord
    with DOW $121323$. }\label{ex121323}
    \end{center}
\end{figure}

\begin{example}\label{121323ex}
{\rm 
In Figure  \ref{ex121323}, boundary curves of two ribbon graphs are shown for the graph representing 
the word $121323$. 
In   %the 
 Figure \ref{ex121323}(A), one sees that the boundary curve is connected, and by Lemma  \ref{Eulerlem}, 
its genus is $2$. In Figure \ref{ex121323}(B), where the connection at vertex $3$ is changed, one sees that the boundary curves 
consist of $3$ components, and hence its genus is $1$. 
In fact we will find that the genus range of this graph is $\{1,2\}$ in Section~\ref{GenusRangeTC}. 
}
\end{example}

\begin{figure}[htb]
    \begin{center}
   \includegraphics[width=6in]{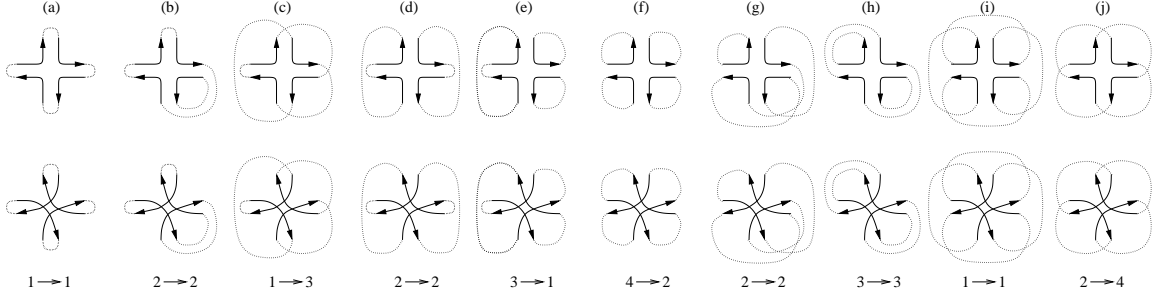}\\
    \caption{Possibilities of strand connections }\label{strand-connection}
    \end{center}
\end{figure}

In Figure  \ref{strand-connection}, all possibilities of global connections of local arcs %%are depicted in the top row
for a neighborhood of a vertex corresponding to Figure  \ref{flip}(A)
are depicted in the top row,  %% moved from above
  and in the middle row, 
the connections after the change of connection as in Figure  \ref{flip}(C) are depicted. 
In the bottom row, the changes of the numbers of boundary components are listed.

\begin{figure}[htb]
    \begin{center}
   \includegraphics[width=4in]{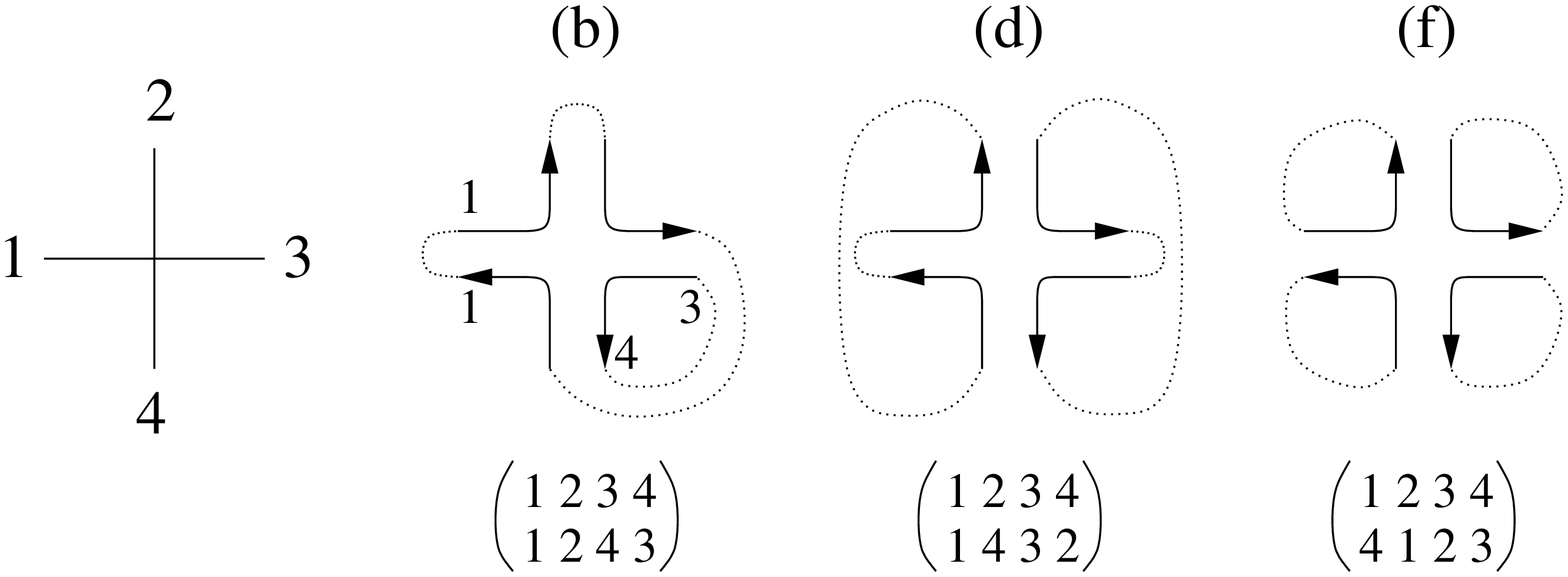}\\
    \caption{Permutations of strand connections }\label{perm-connection}
    \end{center}
\end{figure}

In Figure~\ref{perm-connection}, we illustrate  %indicate 
how
  % we obtain 
  all possible connections
  are obtained.  %%
We label the edges at a vertex by $1, 2, 3, 4$ as in the left of the figure, following the cyclic order of the rigid vertex.
Each edge has two boundary curves, also labeled by the same number,  in opposite orientations. 
For a given ribbon graph, the boundary arcs of each edge are connected away from the neighborhood of the vertex.
The connection is depicted by dotted lines. This  connection by dotted lines  defines a permutation of four letters as follows.
In (b), for example, a dotted line connects the outgoing boundary curve of 
the edge $1$ to the incoming curve of the same edge, so the permutation sends $1$ to $1$. 
The  outgoing boundary curve of 
the edge $4$ is connected by a dotted line 
 to the incoming curve of the edge $3$, so the permutation sends $4$ to $3$. 
 Thus the connection (b) corresponds to the permutation  $(1243)$. 
 By symmetry of $90$ degree rotations, the same connection corresponds to three other permutations 
$(1324)$, $(2134)$, and $(4231)$. Up to symmetry, the connection (d) corresponds to two permutations, and (f)
corresponds to a single permutation. 
Each of the connections (b), (c), (e) and (g) in Figure~\ref{strand-connection} corresponds to four permutations,
each of (d) and (h) corresponds to two, and each of (a), (f), (i) and (j) corresponds to only one permutation. 
This exhausts all 
24   %
permutations.

{}From computer calculations we notice that the genus range always consists of consecutive integers.
Indeed, we prove the following.

\begin{lemma}\label{conseclem}
The  genus range of any assembly graph consists of consecutive integers.
\end{lemma}
\begin{proof}
Let $\Gamma$ be an assembly graph and 
$F$ be a ribbon graph. 
Let $\gr(\Gamma)$ be the genus range of $\Gamma$.
and  let  %Then 
$a$, $b$ be the  %are 
minimum and 
the %%
 maximum integers in $\gr(\Gamma)$,
 respectively.   %%
Let $F, F'$ be the corresponding ribbon graphs
with genus $a$ and $b$,  
respectively. 
Then $F'$ is obtained from $F$  
 by changing the boundary connections at some of the vertices. 
 Thus there is a sequence of ribbon graphs $F=F_0, F_1, \ldots, F_k=F'$
 such that $F_{i+1}$ is obtained from $F_i$ by changing the boundary connection at a single vertex for
 $i=0, \ldots, k-1$. 
 When a connection is changed at one vertex,  the number of boundary components 
changes by at most two as indicated in Fig.~\ref{strand-connection}, where 
the change in the number of boundary components are indicated at the bottom of the figure.
By Lemma~\ref{Eulerlem}, the difference of the genus between two consecutive ribbon graphs
$F_i$ and $F_{i+1}$ is  at most one.
Hence there exists a cellular embedding of $\Gamma$ with genus $c$ for any $c \in \Z$ such that $a < c < b$. 
Thus $\gr(\Gamma)$ consists of consecutive integers.
\end{proof}

\begin{corollary}\label{cplem}
For any $n \in \N$, we have $\GR_{2n-1}, \GR_{2n} \subset {\cal CP}(n)$. 
\end{corollary}
\begin{proof}
This follows from Lemma~\ref{conseclem}, and Lemma~\ref{Eulerlem} since $b(F)\geq 0$. 
\end{proof}

%%%%%%%%%%%%%%%
\section{Properties of Genus Ranges}\label{PropOfGenus}
%%%%%%%%

In this section we investigate properties of genus ranges that 
will be used in later sections. 
Although the following lemma is a special case of Lemma \ref{joinlem}, 
we state it separately  since 
its simple form is convenient to use.

\begin{lemma}\label{looplem}
For any DOW $w$, the corresponding genus range 
of $\Gamma_w$ %%
is equal to that of 
$\Gamma_{w'}$ where %%
  $w'=waa$ and  %where 
   $a$ is a  % single 
    letter
that does not appear in $w$.
\end{lemma}
\begin{proof}
The assembly graph $\Gamma_{w'}$ corresponding to $w'$ is obtained from the graph $\Gamma_w$ for $w$ 
by adding a %%small 
  loop with a single vertex corresponding to the letter $a$. Both of the two 
  boundary   %%
   connections at the added vertex 
   $a$  %%
   increases the number of boundary components by one as depicted in Figure  \ref{loops}(A) and (B). 
Hence the genus range remains unchanged.
\end{proof}

\begin{figure}[h]
\begin{center}
\includegraphics[width=4cm]{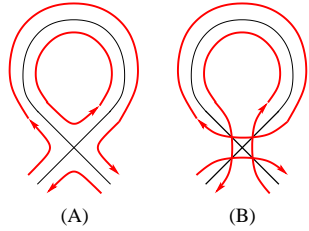} 
\caption{A loop and its connections}
\label{loops}
\end{center}
\end{figure} 

\begin{definition}\label{loopnestdef}
{\rm
A DOW $v$ is said to be obtained from $w$ by {\it loop nesting}  if there exists a sequence  
of DOWs $w=w_0, w_1, \ldots, w_n=v$ such that $w_{i+1}=w_i' a_i a_i$ where $w_i'$ is a cyclic permutation
or reverse % 
of $w_i$ and $a_i$ is a single letter that does not appear in $w_i$ for $i=0, \ldots, n-1$.  

A DOW obtained from the empty word $\epsilon$ by loop nesting is called {\it loop-nested}.
An assembly graph corresponding to a loop-nested DOW is called a {\it loop-nested} graph. 
}
\end{definition}

Lemma \ref{looplem} implies that loop nesting preserves the genus ranges. Specifically,

\begin{corollary}\label{loopcor}
If a DOW $w'$ is obtained from a DOW $w$ by loop nesting, then their 
corresponding assembly graphs have the same genus range.
\end{corollary}

\begin{corollary}\label{higherlem}
If a set $A$ appears as the genus range in ${\cal GR}_n$ for a positive integer $n$, 
then $A$ appears as a genus range in ${\cal GR}_m$
for any integer $m>n$. 
\end{corollary}
\begin{proof}
Let $w'=w a_1 a_1 a_2 a_2 \cdots a_{k} a_k $ where $k=m-n>0$, then by Corollary \ref{loopcor} 
$\Gamma_{w'}$ has the same genus range as $\Gamma_w$. 
\end{proof}

\begin{figure}[htb]
    \begin{center}
   \includegraphics[width=1.5in]{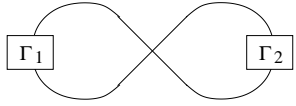}\\
    \caption{Cross sum }\label{crosssum}
    \end{center}
\end{figure}

\begin{definition}\label{joindef}
{\rm
Let $\Gamma_1$ and $\Gamma_2$ be  assembly graphs.
An assembly  graph $\Gamma$ is said to be obtained from $\Gamma_1$ and $\Gamma_2$ by 
a {\it cross sum} 
 if it is formed by 
connecting the two graphs to
the figure-eight graph as depicted in Figure  \ref{crosssum}.
}
\end{definition}

This operation in relation to the number of boundary components was discussed in \cite{JSC}
in relation to virtual knots.
The cross sum construction depends on the choice of edges of the two graphs to
connect to the figure-eight.
Note that the graph $\Gamma$ obtained from $\Gamma_1$ and $\Gamma_2$ by a cross sum  is reducible.

\begin{figure}[htb]
    \begin{center}
   \includegraphics[width=5in]{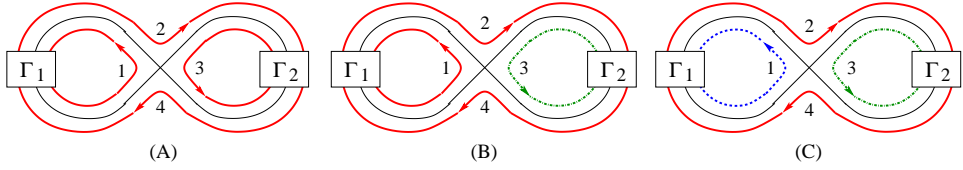}\\
    \caption{Boundary curves of cross sum }\label{crosssumbd}
    \end{center}
\end{figure}

\begin{lemma}\label{joinlem}
Let $\Gamma_1$ and $\Gamma_2$ be  assembly graphs.
If $\Gamma$ is obtained from $\Gamma_1$ and $\Gamma_2$ by a cross sum, then 
 $\gr(\Gamma)=\{ \, g_1 + g_2  \mid  g_1  \in \gr(\Gamma_1), \ g_2 \in \gr(\Gamma_2) \, \} $. 
\end{lemma}
\begin{proof}
For given ribbon graphs $F_1$ and $F_2$ of $\Gamma_1$ and $\Gamma_2$, respectively,
a ribbon graph $F$ of $\Gamma$ is constructed as depicted in Figure  \ref{crosssumbd}. 
The two boundary curves at the edge where $\Gamma_i$ ($i=1,2$) is  connected to the figure-eight
belong to either one component or two in each of $F_1$ and $F_2$. 
In Figure  \ref{crosssumbd}(A), the case where the two curves belong to one component for both of $\Gamma_i$,
$i=1,2$, is depicted. In (B), the case when the two arcs belong to distinct components for one graph
($\Gamma_2$ in this figure) is depicted, and in (C), the case when the two arcs belong to distinct components 
for both  of $\Gamma_i$,
$i=1,2$, is depicted.
{}From the figures we identify these cases (A), (B), (C) with the connection cases 
(a), (b) and (h), respectively, in Figure  \ref{strand-connection}. 
We see from Figure  \ref{strand-connection} that 
by changing the connections of the boundary components in %% added
  all these cases,   %%do not change 
  the number of boundary components 
  does not change.   %% moved 
  Let $v_i$ and $b_i$ be the numbers of vertices and boundary components of $F_i$, $i=1,2$,
respectively, and similarly $v$, $b$ for $F$.
Then 
for all cases (A), (B) and (C)  %% moved from below 
we have  %compute 
$v=v_1+v_2+1$ and $b=b_1+b_2 -1$ %%for all cases (A), (B) and (C). 
and  %Then 
 from  Lemma \ref{Eulerlem} 
we 
compute
$$
g(F) =   \frac{1}{2} ( v- b + 2 )  
= \frac{1}{2} ( (v_1+ v_2 + 1) - (b_1 + b_2 - 1) + 2     ) 
  =g(F_1) + g(F_2) 
$$
as desired.
\end{proof}

\begin{figure}[htb]
    \begin{center}
   \includegraphics[width=1.5in]{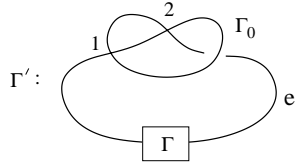}\\
    \caption{Connecting a pretzel }\label{connecttref}
    \end{center}
\end{figure}

%% move def for edge tracing here and slightly reword lemma:
We say that the boundary component $\delta$ of a ribbon graph  $\Gamma$ \textit{traces} the edge $e$ of $\Gamma$ if 
the boundary of the ribbon that contains $e$ is a portion of $\delta$. There are at most two boundary components that can trace an edge.

\begin{lemma}\label{plus1lem}
Let $\Gamma$ be an assembly graph, $\Gamma_0$ be the graph corresponding to the word $1212$, 
and $\Gamma'$ be the graph obtained by connecting an edge $e$ of $\Gamma$ with  %and 
  $\Gamma_0$ as depicted in Figure  \ref{connecttref}.
Suppose $\gr(\Gamma)=[m,  n]$  %m+n]$%% also changed below
  for non-negative integers $m$ and  $n$, ($m\le n$).  %%%

\noindent
{\rm (i)} Suppose for all ribbon graphs of $\Gamma$, 
 the two boundary curves tracing edge % at
   $e$ belong to two distinct boundary components, 
then $\gr(\Gamma')=[m+1,  n+1 ]$.  %%

\noindent
{\rm (ii)} Suppose  for some ribbon graphs of $\Gamma$,
the two boundary curves tracing edge   %at 
 $e$ belong to the same boundary component, 
then $\gr(\Gamma')=[m, n+1 ]$.  %%
\end{lemma}
\begin{proof}
In Figure  \ref{tref}, the boundary curves corresponding to all possible connections at vertices of $\Gamma_0$
are depicted in (A) through (D). 
The number of boundary components are $3$ for (A) and (C), and $2$ for (B) and (D).
In (E) and (F), the situations of how they are connected are schematically depicted,
respectively. The figure (E) represents (A) and (C), and (F) represents (B) and (D).

\begin{figure}[htb]
    \begin{center}
   \includegraphics[width=4in]{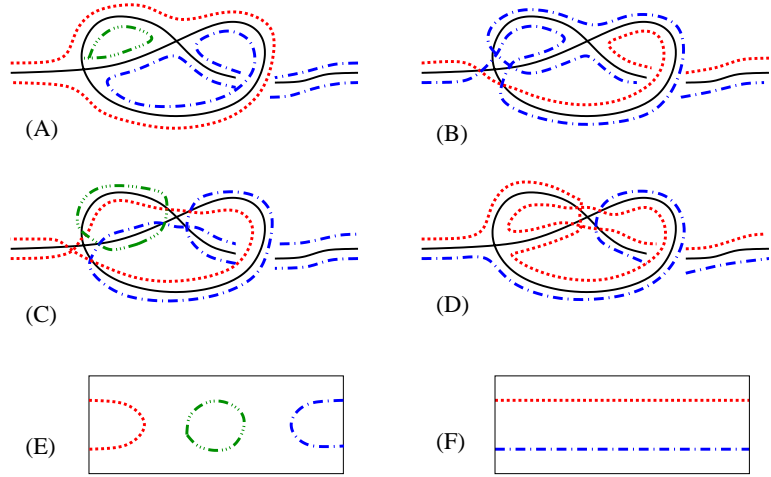}\\
    \caption{ Boundary curves for $1212$ }\label{tref}
    \end{center}
\end{figure}

{\rm (i)}  Consider a ribbon graph $F$ for $\Gamma$ with genus $g \in \gr(\Gamma)$. 
For each ribbon graph $F_0$ in (A) through (D) of $\Gamma_0$ in Figure  \ref{tref}, we obtain a ribbon graph $F'$
of $\Gamma'$ by connecting corresponding boundaries of 
the %
ribbon graphs for $\Gamma_0$ and $\Gamma$. 
If the connection  of the boundary curves of a ribbon graph $F_0$ of $\Gamma_0$ is as in Figure  \ref{tref}(E),
then the two  curves tracing  %at 
  $e$ become connected, 
and there is an additional component indicated by a small circle 
in the middle of (E).  %, so 
 So %that 
  the number of 
boundary components remains unchanged, and  the number of vertices increases by $2$.
Hence $g(F')=g+1$, and $g+1$ is an element of $\gr(\Gamma')$.
If the connection of the boundary curves of  a ribbon graph $F_0$ of $\Gamma_0$ is as in Figure  \ref{tref}(F),  then again
the number of boundary curves  of the ribbon graph for $\Gamma'$ remains the same as $F$, 
and  the number of vertices increases by $2$, so that $g(F')=g+1$,
hence $g+1$ is an element of $\gr(\Gamma')$. 
Therefore  $\gr(\Gamma)=[ m+1,  n+1 ]$.  %%

{\rm (ii)} If the connection  of the boundary curves 
of a ribbon graph $F_0$ of $\Gamma_0$ is as in Figure  \ref{tref}(E),
then the single  %same 
 boundary component 
 tracing %at 
  $e$ becomes disconnected, and there is an additional component
  in $F'$.  %% 
 Therefore the number of 
boundary components  increases by $2$, and the number of vertices increases by $2$. 
Hence $g(F')=g$, and $g$ is an element of $\gr(\Gamma')$.
If the connection  of the boundary curves 
of a ribbon graph $F_0$ of $\Gamma_0$ is as in Figure  \ref{tref}(F),  then 
the number of boundary components of the ribbon graph for $\Gamma'$ remains the same as that of $F$.
In this case  the number of vertices increases by $2$, so $g(F')=g+1$,
and $g+1$ is an element of $\gr(\Gamma')$. 
Therefore  $\gr(\Gamma)=[ m,  n+1] $. %%
\end{proof}

%%%%%%%% moved above
%We say that the boundary component $\delta$ of a ribbon graph  $\Gamma$ \textit{traces} the edge $e$ of $\Gamma$ if 
%the boundary of the ribbon that contains $e$ is a portion of $\delta$. There are at most two boundary components that 
%can trace an edge.
As to the condition (i) in Lemma \ref{plus1lem}, we note the following.

\begin{lemma}\label{planarlem}
Let $\Gamma$ be a planar assembly graph and $e$ be an edge in $\Gamma$. 
Then  for all 
 ribbon graphs of $\Gamma$, 
there are two distinct 
boundary %% 
components  %%of  the  boundary curves 
 that trace $e$.
\end{lemma}
\begin{proof}
Fix a planar diagram of $\Gamma$, still denoted by $\Gamma$. 
The closure of a topological, thin  neighborhood of $\Gamma$ in the plane determines a ribbon graph $F$.
Each  boundary curve of $F$ corresponds to a region of the complement of $F$
(a connected component of $\R^2 \setminus F$)  in the plane. 
In this construction, the connection of boundary curves at every vertex 
has the form in  Figure~\ref{flip}(A).
We assign checker-board ``black" and ``white" coloring %%  added
  to these regions (there is one, see  \cite{Kauff} for example)
  such that regions bordering the same edge have distinct colors.  %% added
The corresponding boundary curves of $F$ inherit this coloring. 
Hence every edge is traced by boundary components of two distinct colors. 
Let $B$ and $W$ be (disjoint, non-empty) families of boundary curves colored black and white, respectively.
Any ribbon graph $F'$ of $\Gamma$ is obtained from $F$ by changing boundary connections at some vertices. 
Let $v_1, \ldots, v_k$ be a sequence of vertices at which connections are changed in this order 
to obtain $F'$ from $F$, and let $F_i$ be the ribbon graph obtained from $F$ by changing boundary connections at $\{ v_1, \ldots, v_i \}$, $1 \leq i \leq k$. 
When we change the boundary connection at $v_1$ and obtain $F_1$, 
two diagonal boundary curves of the same color become connected.
This corresponds to the connection $(f)$ in Figure~\ref{strand-connection}. 
In Figure~\ref{flip}(A), the curves of the pair $1$ and $3$, and the pair $2$ and $4$ are of the same colors.
The connected curves inherit  %s 
 the color, and each component of  the boundary curves of $F_1$ is colored black or white, 
 and every edge of $\Gamma$ is traced by two boundary components of $F_1$, one white and one black. %%added
Let $B_1$ and $W_1$ be  families of boundary curves 
of $F_1$  %% add
 colored black and white, respectively. They are disjoint and non-empty.
  %For $F_1$, every edge is traced by curves of distinct colors. 
Inductively, suppose that $F_i$ has two families of disjoint and non-empty boundary curves $B_{i}$ and $W_i$,
and that every edge of $F_i$ is traced by curves from both families. 
At $v_{i+1}$, the boundary curves are as depicted in 
 Figure~\ref{flip}(A), where the curves of the pair $1$ and $3$, and the pair $2$ and $4$ are of the same colors,
respectively, and by changing the connection to (C), the curves of the same colors are affected diagonally. 
Then the boundary curves of $F_{i+1}$ inherit the colors from those of $F_i$ to form two families $B_{i+1}$ and $W_{i+1}$,
for $i=1, \ldots, k-1$, and every edge of  $F_{i+1} $ is traced by curves of distinct colors. 
This concludes that the %The  
 boundary curves at each edge of $F_k$ have two distinct colors, and hence they belong to distinct boundary components. 
\end{proof}

\begin{lemma}\label{fullrangelem}
For any positive integer $n$, there is no assembly graph with $2n-1$ vertices 
with genus range $[0, n]$. 
\end{lemma}
\begin{proof}
Suppose there is such an assembly graph $\Gamma$.
Since its genus range contains $0$, $\Gamma$ is planar.
By Lemma~\ref{planarlem}, every ribbon graph of $\Gamma$ has more than one boundary component, 
and therefore, the genus cannot be $n$ by Lemma~\ref{Eulerlem}.
\end{proof}

\begin{lemma}\label{notnnlem}
For any positive integer $n$, there is no assembly graph with $2n-1$ vertices 
with genus range $[n, n]$. 
\end{lemma}
\begin{proof}  We  % First we 
  observe that for every assembly graph there is a ribbon graph with more than one boundary component. Then the Lemma follows immediately. Let $\Gamma$ be an assembly graph with $k$ vertices. Starting from one vertex, 
we enumerate the edges of $\Gamma$ along the transversal from $1$ to $2k$. Since the transversal consists of an Eulerian path in which two consecutive edges are not neighbors, every vertex in $\Gamma$ is incident to two even numbered edges and two odd numbered edges. Moreover, the odd numbered (similarly, even numbered) edges incident to a common vertex are neighbors to each other. Let $G$ be the subgraph of 
$\Gamma$ that is induced by the odd numbered edges. This graph consists of $k$ vertices and $k$ edges, so it is a collection of cycles. Consider one cycle $v_1,v_2,\ldots, v_s$ of $G$.  Let $e_1,e_2,\ldots,e_s$ be the edges of this cycle such that 
$v_i$ is incident to $e_{i-1}$ and $e_i$ 
for $i=2, \ldots, s$, and $v_1$ is incident to $e_1$  and $e_s$.  
Consider 
the ribbon graph $F$ of $\Gamma$ obtained in the following way. Choose an arbitrary connection of the boundary components at vertex $v_1$. One of the boundary components
 at $v_1$, call it $\delta$, traces both $e_s$ and $e_1$. Then at $v_2$ we choose a connection of the
  boundary components such that $\delta$ traces both $e_1$ and $e_2$. One of the connections (A) or (C) in 
  Figure~\ref{flip} provides this condition. Inductively, suppose the boundary connections at
   $v_1,v_2,\ldots,v_i$ have been chosen such that $\delta$ traces $e_s,e_1,\ldots, e_i$. 
   Then at $v_{i+1}$ we choose the boundary connection such that $\delta$ traces $e_i$ and $e_{i+1}$. 
   For $i=s$, we have that $\delta$ traces all edges $e_1,\ldots,e_s$ of the cycle. Note that $\delta$ must 
   trace $e_s$ only once, i.e., it `closes' on itself, because the other boundary component tracing $e_s$ %also 
   traces also %%
   an even-numbered edge incident to $v_1$, the other neighbor of $e_s$. Since $\delta$ traces edges only in $G$, the ribbon graph $F$ must contain boundary components that trace the edges not in $G$, hence, 
   $F$ has more than one component. 
\end{proof}

%%%%%%%%%%
\section{Realizations of Genus Ranges}\label{Realization}
%%%%%%%%%

In this section we construct graphs that realize some desired sets of genus ranges.

\begin{figure}[htb]
    \begin{center}
   \includegraphics[width=3in]{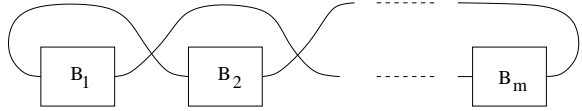}\\
    \caption{ Repeating cross sum}\label{connectall}
    \end{center}
\end{figure}

\begin{proposition}\label{nm-prop} %% changed, check 
For any integer $m\geq 0$, there exists an integer $n$ such that  $\{ m \}$
is a genus range in ${\cal GR}_n$. 
\end{proposition}
\begin{proof}
Let $\Gamma_m$ be the graph obtained by connecting $m$ copies of $\Gamma_0$ in the boxes $B_1, \ldots, B_m$
as depicted by solid line in Figure  \ref{connectall}, where $\Gamma_0$ is the graph depicted  in the upper half of Figure  \ref{connecttref} 
 that corresponds to the word $1212$.
This is obtained from copies 
of  %
 $\Gamma_0$ by repeated application of cross sum of Definition \ref{joindef}.
Recall that $\Gamma_0$ has the genus range $\{ 1 \}$ (Remark~\ref{highest-single-rem} Item 1). 
By Lemma \ref{joinlem}, $\Gamma_m$ has the genus range $\{ m \}$.
\end{proof}

\begin{remark}\label{gr-0-lem}
{\rm
 A DOW $w$ is loop-nested if and only if  $\gr(\Gamma_w) = \{ 0 \}$. 
Corollary \ref{loopcor} implies that if $w$ is loop-nested then  $\gr(\Gamma_w) = \{ 0 \}$.
The converse  follows from a known characterization of signed
DOWs corresponding to closed normal planar curves (Theorem $1$ of \cite{CE1}).
}
\end{remark}

\begin{figure}[htb]
    \begin{center}
   \includegraphics[width=2.5in]{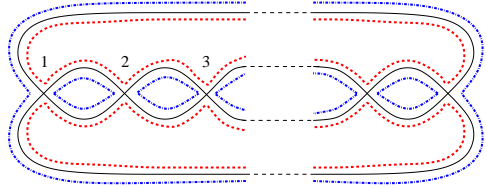}\\
    \caption{Graphs for repeat words }\label{repeat}
    \end{center}
\end{figure}

\begin{proposition}\label{01prop}
The graph corresponding to  any word obtained from 
$w=12\cdots n 12 \cdots n$,  for an odd integer $n$, by loop nesting 
has the genus range $\{0, 1\}$. 
\end{proposition}
\begin{proof}
First we show that the graph $\Gamma_n$ corresponding to  a  word 
$w=12\cdots n 12 \cdots n$ has genus range $\{0,1\}$.
We start with the planar diagram depicted in Figure  \ref{repeat}.
Boundary curves of the ribbon graph obtained as a neighborhood of $\Gamma_n$ are
depicted by dotted lines. 
When the connection of  one vertex (the vertex $1$ in Figure  \ref{repeatswitch} left) is changed,
both pairs of diagonal curves are connected from two components to one
(see Figure~\ref{strand-connection}(f)). 
Hence the number of  boundary components reduce by $2$, and the genus increases 
from $0$ to $1$ by Lemma \ref{Eulerlem}.

\begin{figure}[htb]
    \begin{center}
   \includegraphics[width=5in]{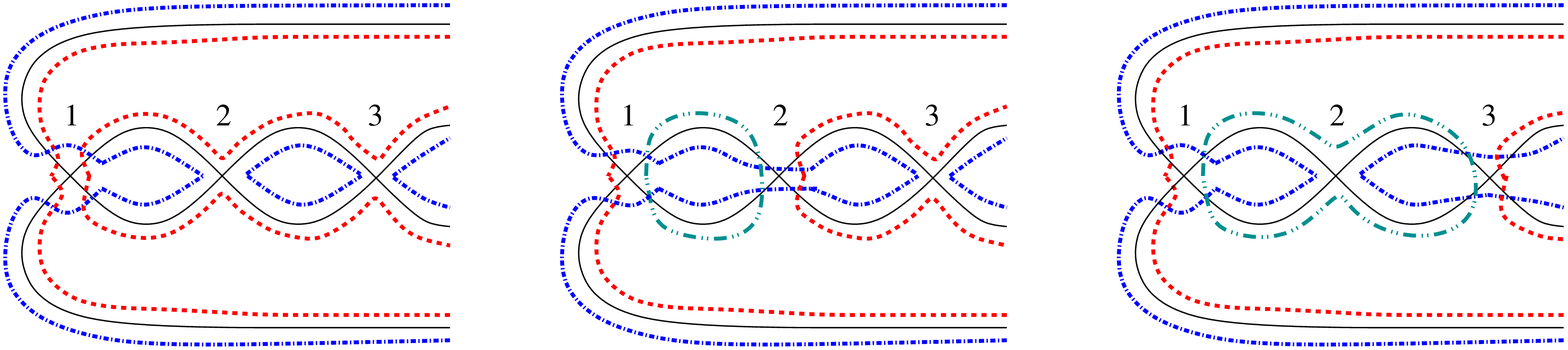}\\
    \caption{Changing connections for  repeat word graphs }\label{repeatswitch}
    \end{center}
\end{figure}

After the change at the vertex $1$, the boundary connection at every other vertex 
$2$ through $n$ is as depicted in Figure~\ref{strand-connection}(h).
Hence the number of the boundary components does not change by changing the boundary connections.
After the second change of the boundary connection at the vertex $i$, the connection at every vertex 
remains the same as in Figure~\ref{strand-connection}(h).
This property holds for odd $n$.
In Figure  \ref{repeatswitch} middle and right, changes are made at vertices $i=2$ and $3$, respectively, 
after the change at the vertex $1$. 

The statement follows from Lemma~\ref{looplem}.
\end{proof}

\begin{conjecture}
{\rm
Any DOW whose corresponding graph has  genus range $\{0, 1\}$ is 
obtained from the word $w=12\cdots n 12 \cdots n$ for an odd integer $n$ by loop nesting.
}
\end{conjecture}

\begin{figure}[ht]
\begin{minipage}[b]{0.3\linewidth}
\centering
\includegraphics[scale=.29]{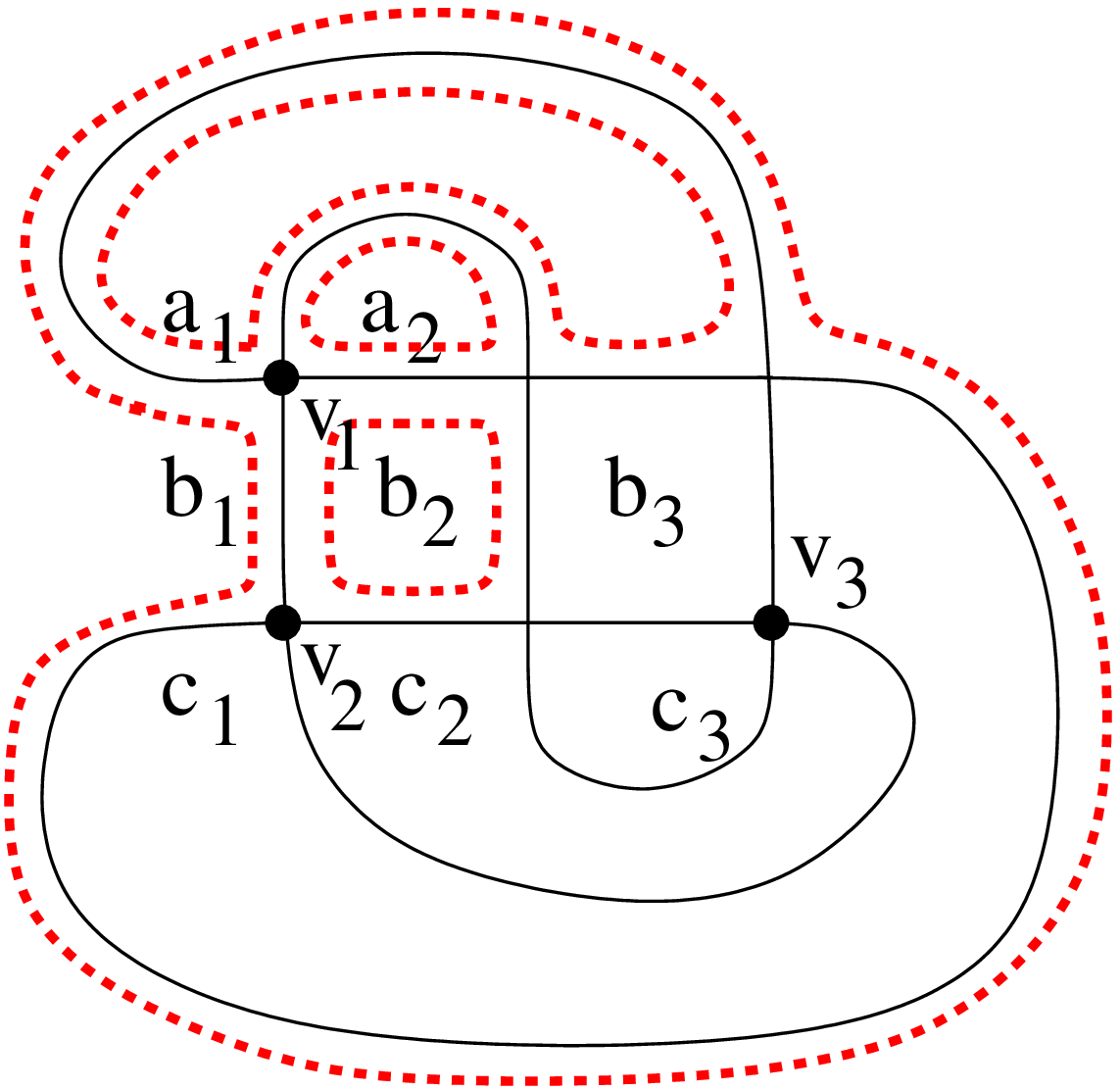}
\caption{Proof for $P_3$}
\label{n6step1}
\end{minipage}
\begin{minipage}[b]{0.3\linewidth}
\centering
\includegraphics[scale=.2]{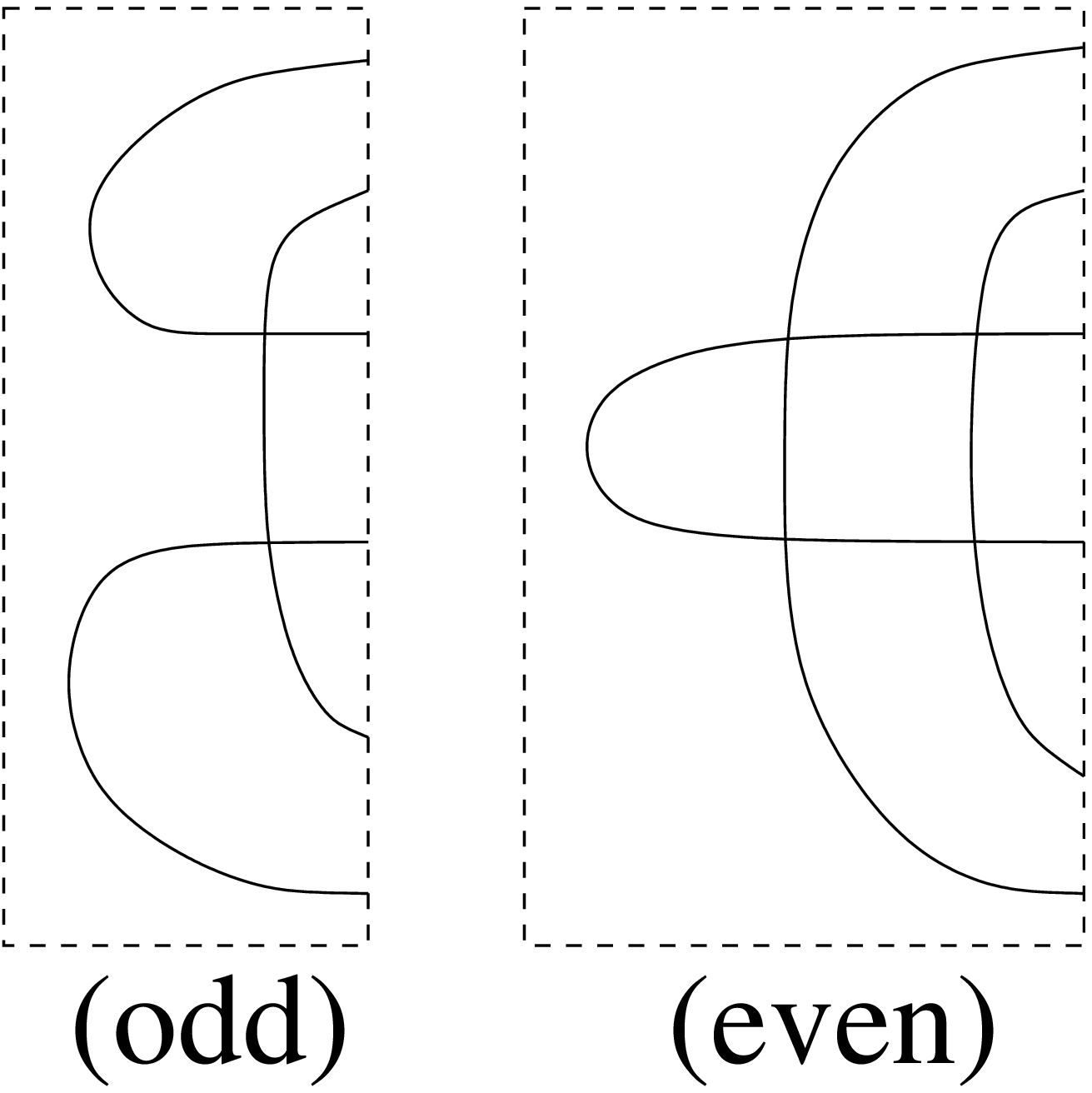}
\caption{Box  (X) }
\label{leftbox}
\end{minipage}
\begin{minipage}[b]{0.3\linewidth}
\centering
\includegraphics[scale=.19]{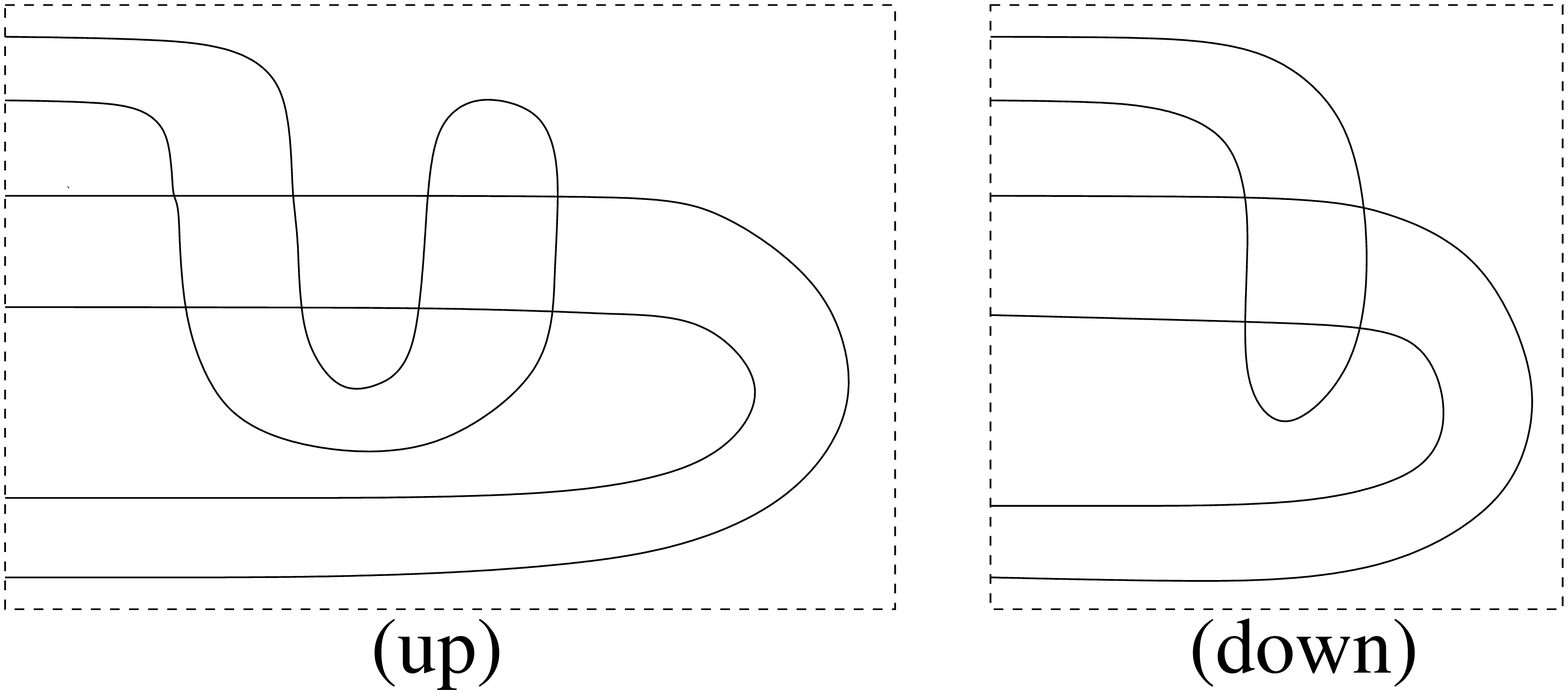}
\caption{Box  (Z) }
\label{rightbox}
\end{minipage}

\end{figure}

\begin{proposition}\label{halfgrprop}
For any  integer $n>1$, 
 there exists an 
assembly graph
$\Gamma$ of $2n$ vertices with $\gr(\Gamma)=\{0, 1, \ldots, n \}$.
\end{proposition}
\begin{proof}
Computer 
calculations show that there is a unique graph of $4$ vertices corresponding to 
$12314324$ that has genus range $\{0,1,2\}$ as mentioned in Remark~\ref{rem-full-range}.   %before.
In Figure  \ref{n6step1}, an assembly graph $P_3$  is depicted that has the genus range $\{0, 1, 2, 3\}$.
To show this, we start with the planar diagram as depicted. 
We consider the ribbon graph that is a thin neighborhood of this planar diagram.
For this ribbon graph, the boundary connection at every vertex appears as type (A) in Figure \ref{flip}. 
The genus of this  ribbon graph is $0$.
Each region of the diagram corresponds 
to a boundary curve of the corresponding ribbon graph.
 Four of such boundary curves (out of total $8$) are depicted by dotted lines 
in the figure, for the regions labeled by $a_1$, $a_2$, $b_1$, and $b_2$.

We change the boundary connection at the vertex labeled by $v_1$
from type (A) in Figure~\ref{flip} to type (C), where the connection is  
as indicated in Figure  \ref{strand-connection}(f).  
Then as Figure  \ref{strand-connection}(f) 
 indicates, the curves  $a_1$ and $b_2$, $a_2$ and $b_1$, are connected, 
respectively, reducing the number of boundary components 
by $2$. 
This causes the increase of the genus of the new ribbon graph by $1$
by Lemma \ref{Eulerlem}, so that the new surface has
genus $1$. 
After the connection change at $v_1$, the boundary connection at $v_2$ remains 
as in  Figure  \ref{strand-connection}(f), since the regions $b_1$, $b_2$, $c_1$, $c_2$ 
are bounded by distinct curves. 
Thus by repeating this process with vertices $v_2$ and $v_3$, 
we obtain ribbon graphs of genus $2$ and $3$, respectively. The number of boundary curves after performing the changes at $v_1$, $v_2$ and $v_3$ is $2$. 
The two components can be seen as a checkerboard coloring of the regions,
and the two boundary curves near every edge belong to distinct components  (cf. Lemma \ref{planarlem}). 
This shows that the genus range for this graph is $\{0,1,2,3\}$.

\begin{figure}[htb]
    \begin{center}
   \includegraphics[width=3in]{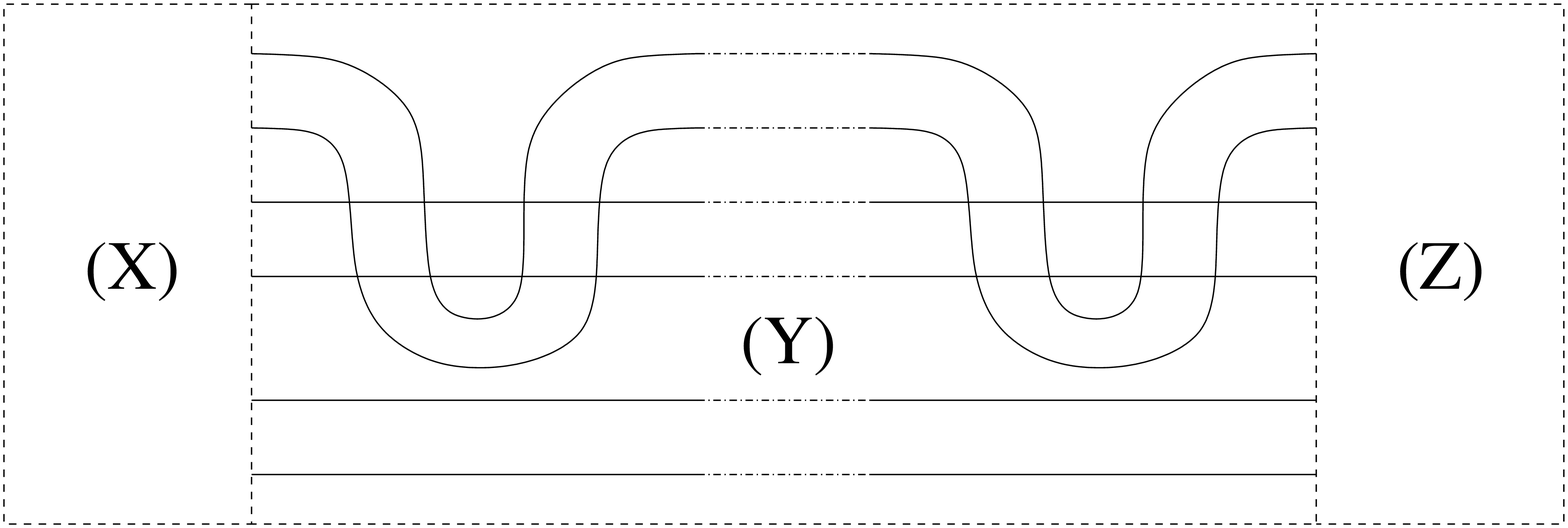}\\
    \caption{Inductive construction of $P_n$ }\label{snake}
    \end{center}
\end{figure}

Next we describe the assembly graph $P_n$ for $n>3$
as the 
 combination of three subgraphs 
 indicated in Figure \ref{snake} as follows.
If $n$ is odd, then the subgraph in the box (X) in  Figure \ref{snake}
is as  depicted in Figure \ref{leftbox}(odd), and if $n$ is even,
 it is as depicted in Figure \ref{leftbox}(even), respectively. 
The middle part in the box (Y) is as  depicted in  Figure \ref{snake}.
The subgraph in box (Z) in   Figure \ref{snake} is either (down) or (up) of 
  Figure \ref{rightbox}, and the choice is determined by the number of vertices
  $|P_n|=2n$. Specifically, for integers $k$,  $P_n$ has the following patterns in boxes (X) and (Z) in    Figure \ref{snake}, respectively:
If $n=4k$, $4k+1$, $4k+2$, $4k+3$, respectively, the box (X) and (Z) are 
filled with (even) and  (down), (odd) and (up), (even) and (up), (odd) and (down), 
respectively. Examples for $n=6$ (even, up) and $7$ (odd, down) are depicted 
in Figure \ref{n8} and \ref{n14}, respectively.

\begin{figure}[ht]
\begin{minipage}[b]{0.5\linewidth}
\centering
\includegraphics[scale=.25]{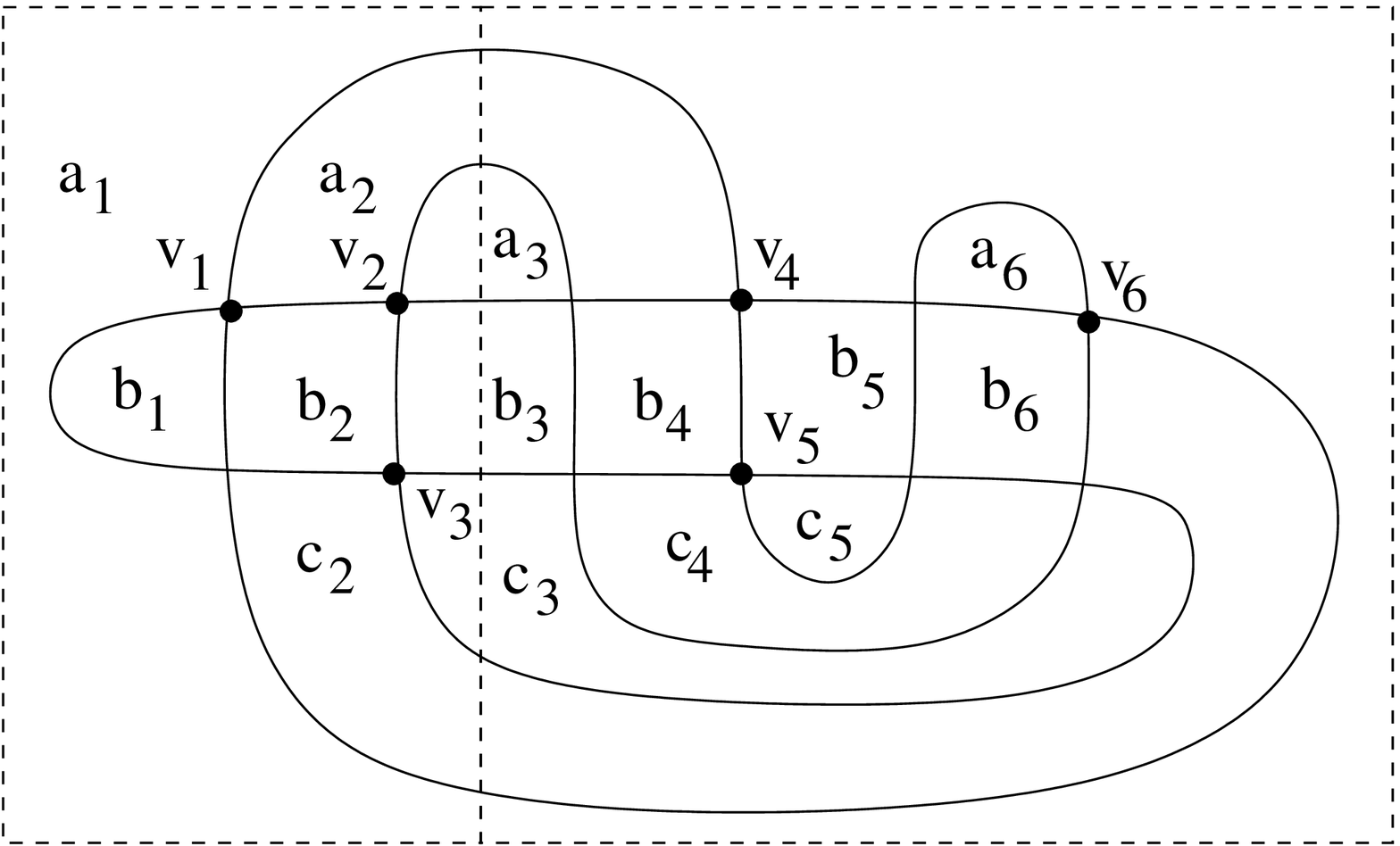}
\caption{$P_6$}
\label{n8}
\end{minipage}
\begin{minipage}[b]{0.5\linewidth}
\centering
\includegraphics[scale=.28]{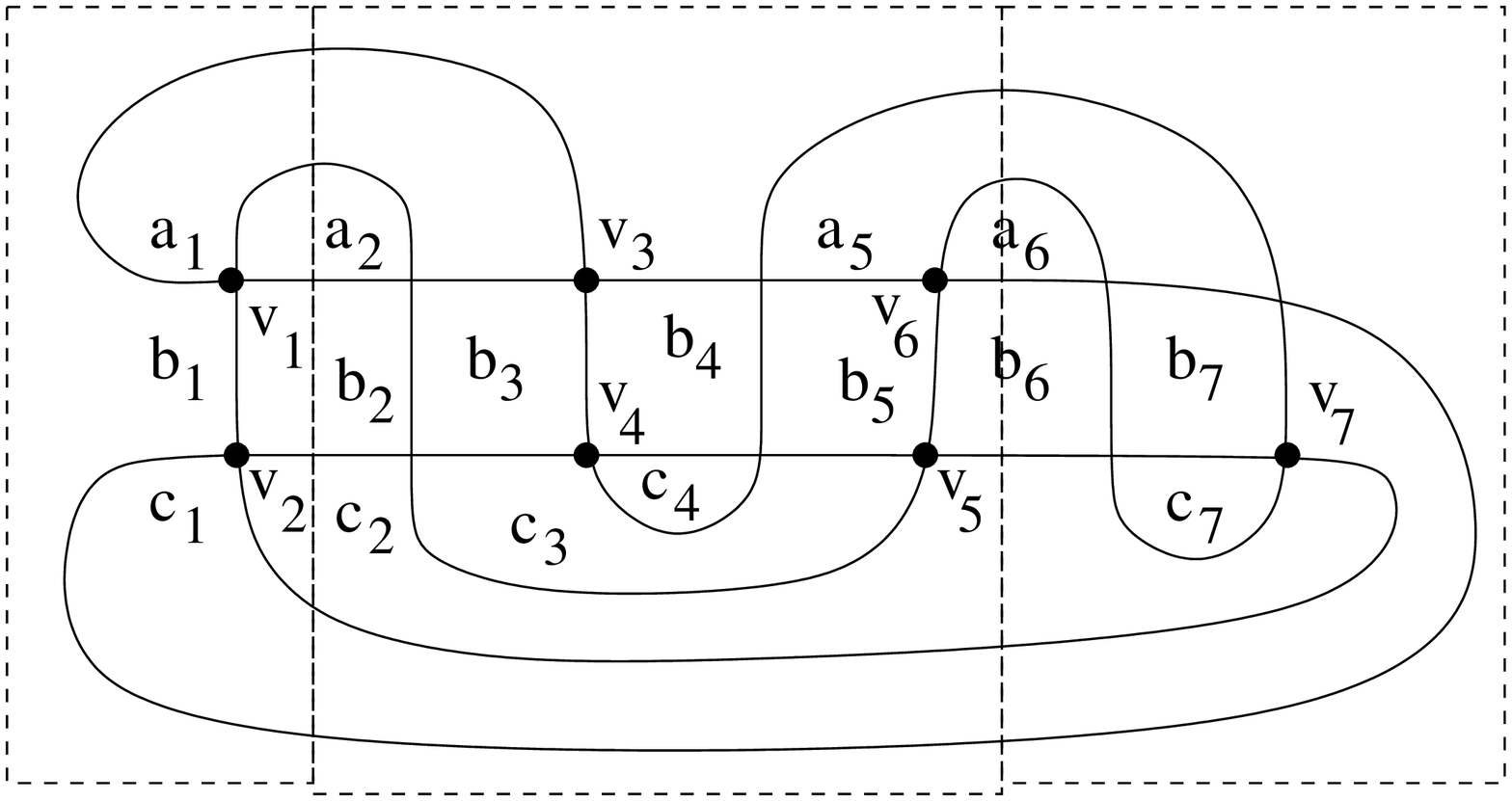}
\caption{$P_7$}
\label{n14}
\end{minipage}
\end{figure}

We generalize the argument for $P_3$ to $P_n$ inductively.
We start with the ribbon graph that is a thin neighborhood of the planar diagram as
described in Figure \ref{snake}, where the boundary connection at every vertex is
of type (A) in Figure \ref{flip}. The connection of boundary curves is as Figure \ref{strand-connection}(f) at every vertex.
We successively change boundary connections 
from type (A) to type (C) at a subset of  vertices we describe below.
With every such change, the number of boundary components of the corresponding ribbon graph reduces by $2$, and hence the genus increases by $1$. 
We start by changing the boundary connections for the vertices in the subgraph in 
box (X).
For  $n$ odd,  we change the boundary connection at 
vertices $v_1$ and $v_2$ in Figure \ref{leftshade}(odd), 
and for $n$ even,  at 
vertices $v_1$, $v_2$ and $v_3$ in Figure \ref{leftshade}(even). 
The checkerboard coloring in Figure \ref{leftshade} indicates that the boundary curves of the 
regions of the same shading are connected.

\begin{figure}[ht]
\begin{minipage}[b]{0.26\linewidth}
\centering
\includegraphics[scale=.28]{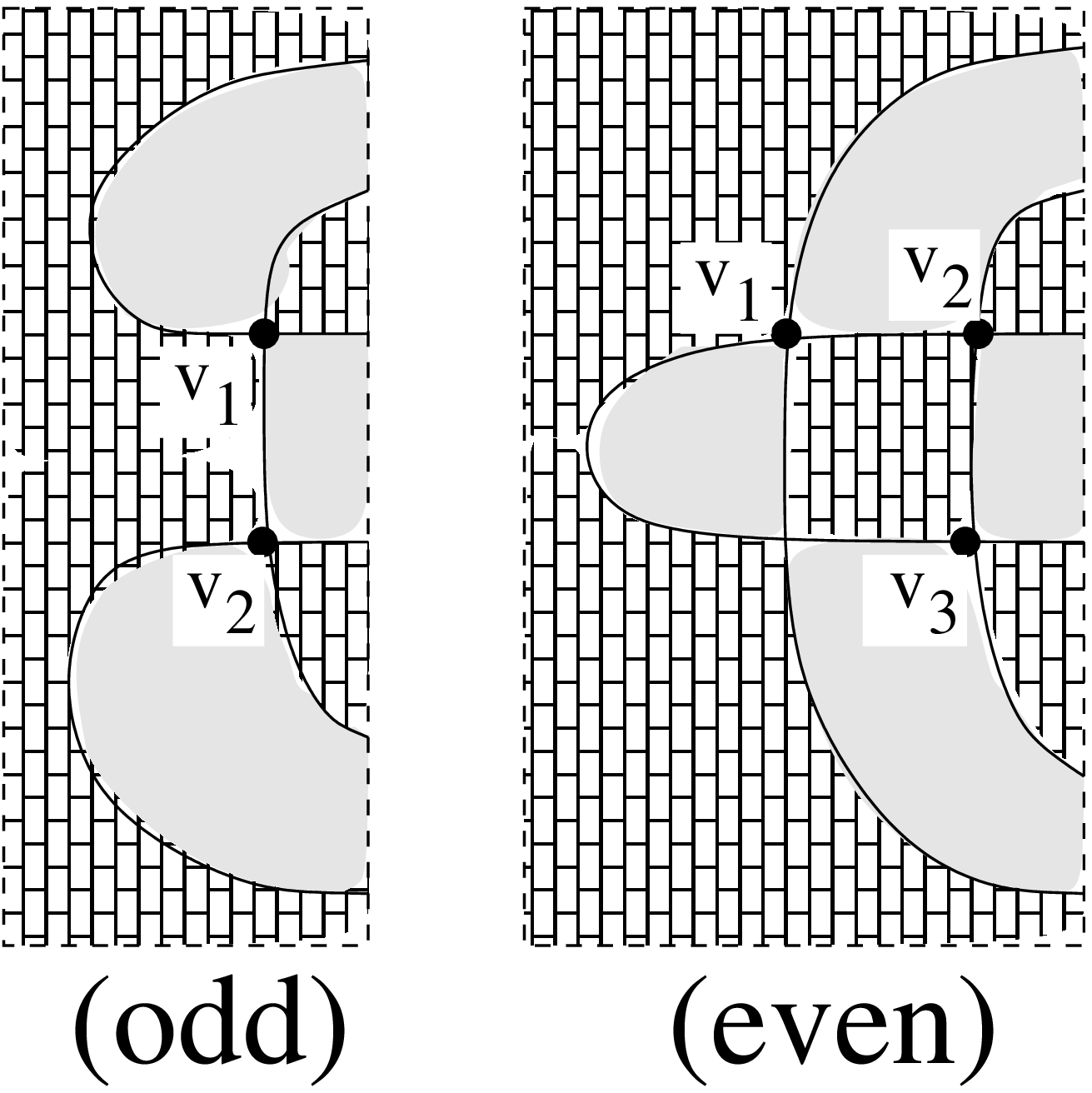}
\caption{Shading  (X) }
\label{leftshade}
\end{minipage}
\begin{minipage}[b]{0.24\linewidth}
\centering
\includegraphics[scale=.28]{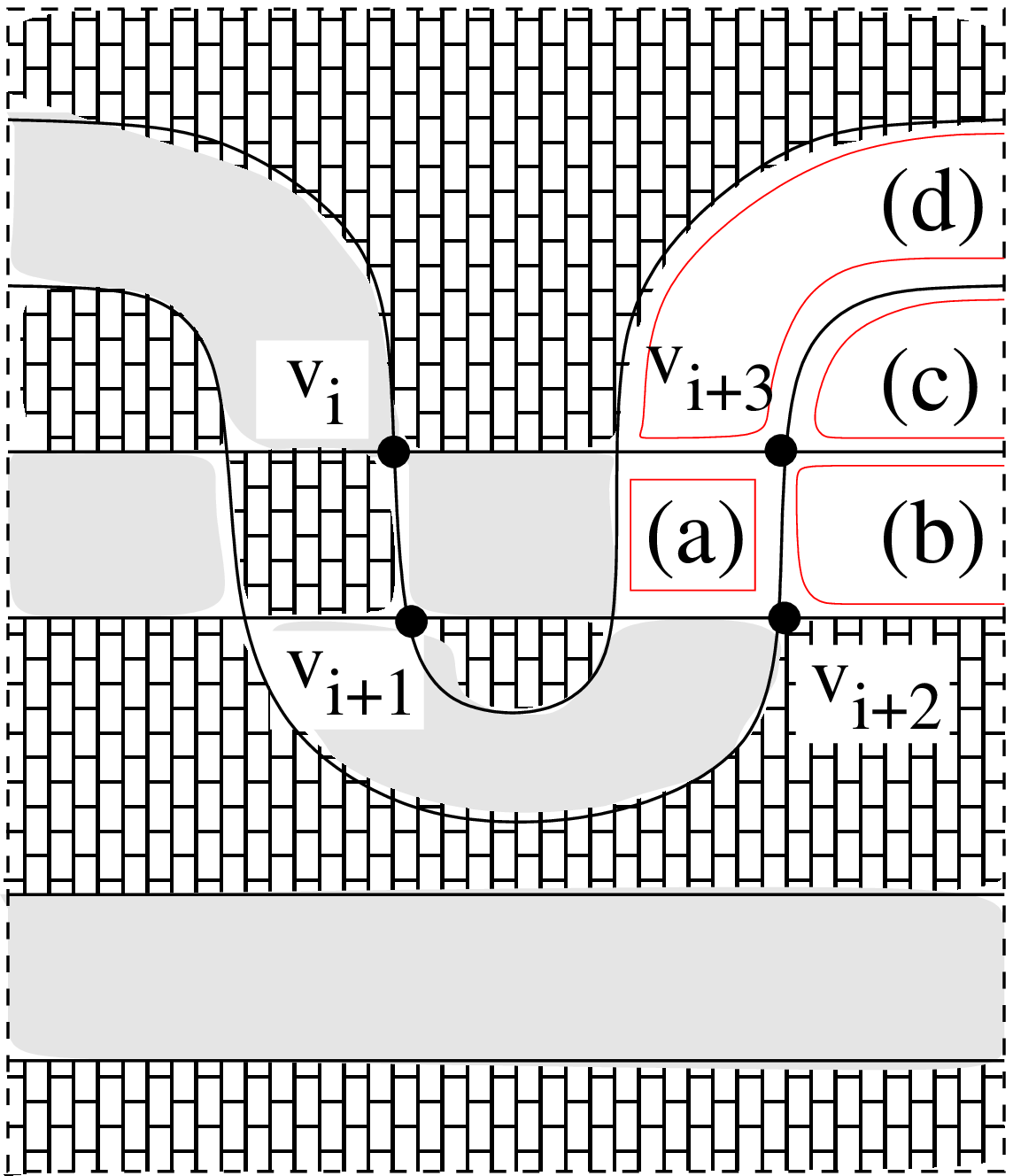}
\caption{Shading  (Y) }
\label{middle}
\end{minipage}
\begin{minipage}[b]{0.35\linewidth}
\centering
{\includegraphics[scale=.25]{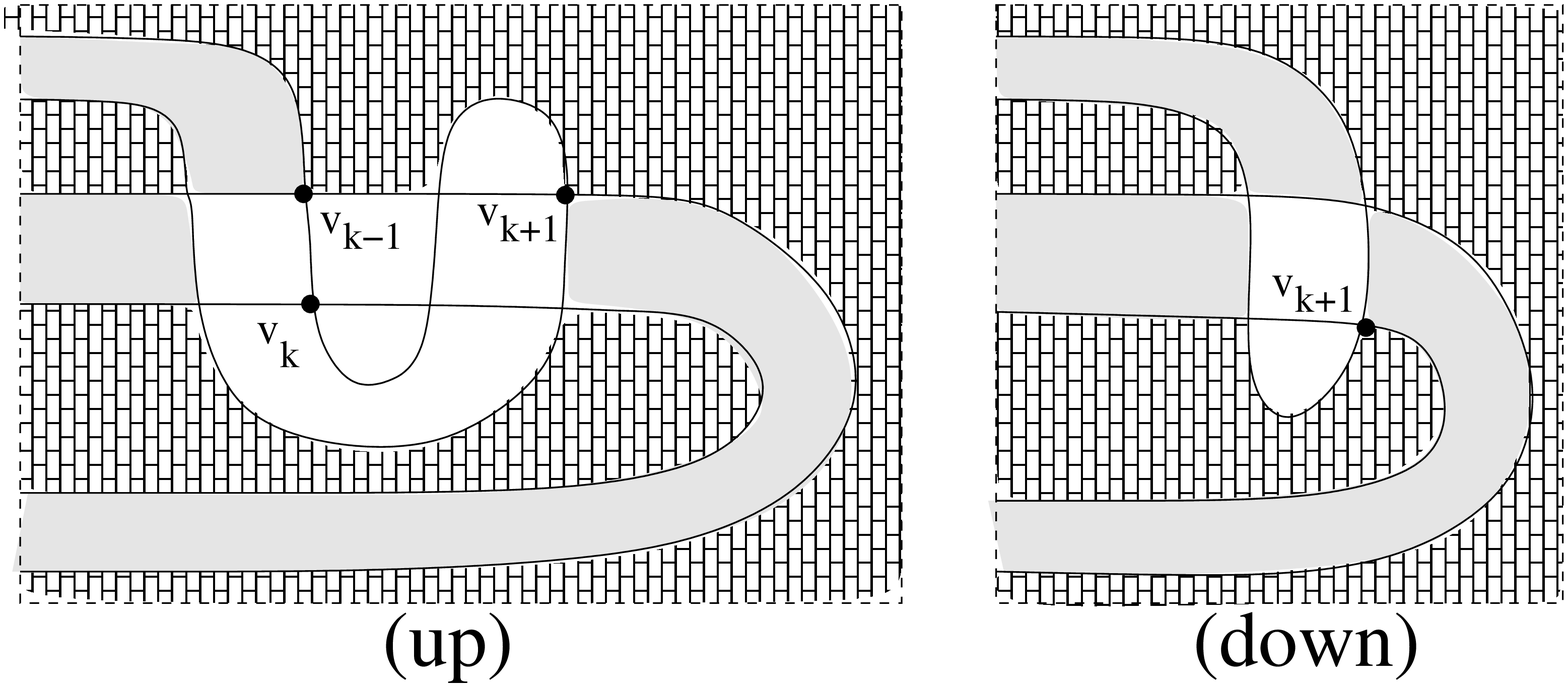}
\caption{Shading (Z)}}
\label{induction16}
\end{minipage}
\end{figure}

Next we proceed to box (Y) inductively. In Figure \ref{middle}, after the changes of the boundary connections at vertices $v_i$ and $v_{i+1}$  in this order,
the boundary curves are connected 
for the curves of the regions of the same shading in Figure \ref{middle}.
Note that the boundary curves of the four regions (a) through (d) in the figure 
belong to distinct components. Hence the connection changes at the vertices 
$v_{i+2}$ and $v_{i+3}$ in this order 
connects the curves (a) and (c), (b) and (d) to those two curves 
corresponding to the two shadings, respectively. 
Inductively, by changing boundary connections in all thickened 
vertices in Figure \ref{snake} box (Y), all curves bounding the regions in box (Y) are connected to two boundary components. The two components are indicated by the checkerboard shading.

We complete the induction in box (Z) as indicated in Figure \ref{induction16}.
The last sequences of vertices where the  boundary connections
are changed  from type (A) to (C) are 
depicted in the figure for each case (up) and (down) of Figure \ref{rightbox}. 
 \end{proof}

\begin{proposition}\label{oddhalfgrprop}
For any  integer $n>1$,
 there exists an 
assembly graph
$\Gamma$ of $2n+1$ vertices with $\gr(\Gamma)=\{0, 1, \ldots, n \}$.
\end{proposition}
\begin{proof}
This follows from Proposition \ref{halfgrprop} and Lemma \ref{higherlem}.
\end{proof}

%%%%%%%%%%%%%%%%%%%%%%%%%%%%%%%%%%%%%
\section{Genus Range of the Tangled Cord}\label{GenusRangeTC}
%%%%%%%%%%%%%%%%%%%%%%%%%%%%%%%%%%%%%%

In this Section we find the genus range of a special family % type 
of assembly graphs
called the tangled cords. In view of 
Problem \ref{MainProblem}, 
we prove that for all $n$ the maximum 
genus range (at least for odd $n$) %% added
$\{ \lfloor \frac{n-1}{2}\rfloor , \lfloor \frac{n+1}{2}\rfloor\}$ is achieved by the tangled cord with $n$ vertices.
The analysis of the genus range in this case uses a technique of  {\it addition}  and {\it removal} of a vertex.

%%%%%%%%%%%%%%%%%%%%%%%%%%%%
\subsection{Addition and Removal of a Vertex}
%%%%%%%%%%%%%%%%%%%%%%%%%%%%

Let $\Gamma$ be an assembly graph and $e, e'$ be two edges in $\Gamma$ with endpoints $v_1,v_2$ and 
$v_1',v_2'$ respectively. 
We consider the graph $\Gamma_{\rm split}(e, e')$ obtained from $\Gamma$ by adding $2$-valent vertices 
$v$ and $v'$ on $e$ and $e'$, respectively, splitting $e$ and $e'$ into two edges $e_1, e_2$ and $e_1', e_2'$
as depicted in Figure~\ref{add-remove}(middle).
The new edges $e_1, e_2, e_1', e_2'$ have end vertices $\{v_1,v\},\{v_2,v\}, \{v_1',v'\}, \{ v_2', v'\}$, 
respectively. 
We say that $\Gamma'$ is obtained from $\Gamma$ by {\it addition of a vertex
by crossing $e,e'$ in cyclic order 
 $e_1,  e_1',e_2, e_2'$} (or simply {\it by addition of a vertex} when context allows) if $V(\Gamma')=V(\Gamma)\cup\{w\}$ and the edges of $\Gamma'$ are obtained from the edges of $\Gamma_{\rm split}(e,e')$ %added subscript
by identifying $v$ and $v'$ to a single vertex $w$ with cyclic order of edges $e_1, e_1', e_2,  e_2'$
as depicted in Figure~\ref{add-remove} from left to right. 
  The cyclic order of the edges at vertices $v_1,v_2,v_1',v_2'$ remains as in $\Gamma$ such that the roles of $e,e'$ are taken by the corresponding new edges. 
  If $\Gamma'$ is obtained from $\Gamma$ by crossing $e,e'$ in  cyclic order $e_1, e_1', e_2,  e_2'$,
 we write $\Gamma'=\Gamma(e_1, e_1', e_2,  e_2')$.
In this case we also say that {\it  $\Gamma$ is obtained from $\Gamma(e_1, e_1', e_2,  e_2')$ by vertex removal}.
Note that the vertices $v_1, v_2, v_1', v_2'$ need not to be distinct.

\begin{figure}[htb]
    \begin{center}
   \includegraphics[width=5.3in]{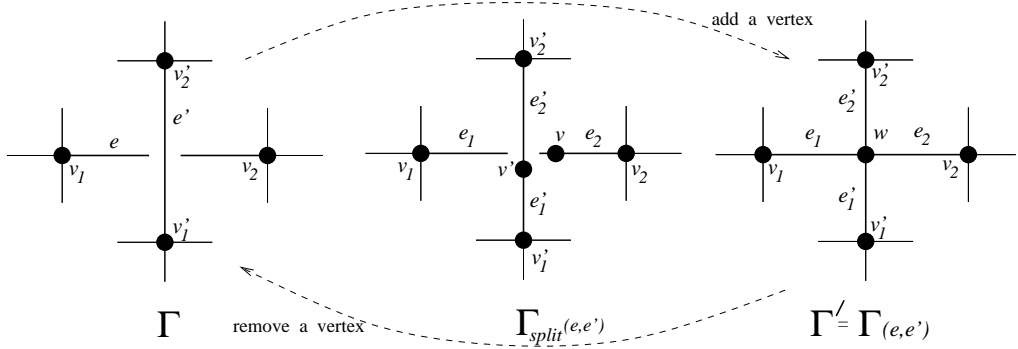}\\ 
    \caption{Addition and removal of a vertex }\label{add-remove}
    \end{center}
\end{figure}

Let $F$ be a  ribbon graph of $\Gamma$  and  let $\Gamma'$ be obtained from $\Gamma$ by addition of a vertex $w$.  A  ribbon graph $F'$ of $\Gamma'$ whose boundary connections at every vertex $v\not =w$ is the same as the 
the boundary connections of  $F$ at $v$ is denoted $F(w)$. Since there are two possible connections of the boundary components of $F(w)$ at the new vertex $w$, the notation indicates one of those two choices.

All possible changes in the global connections of boundary components of $F(w)$ at the new vertex $w$ relative 
to %
the boundary components of  the crossing edges that add the new vertex $w$ are depicted in Figure  \ref{add-vertex-strands}.
Note that the bottom row coincides with the top row of Figure  \ref{strand-connection}.
\begin{figure}[htb]
    \begin{center}
   \includegraphics[width=6in]{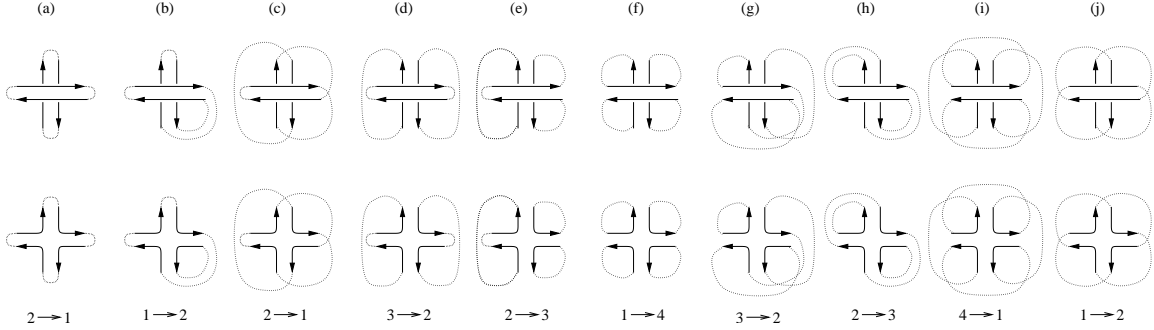}\\
    \caption{Possibilities of strand connections }\label{add-vertex-strands}
    \end{center}
\end{figure}

Recall that %We say that 
the boundary component $\delta$ of a ribbon graph  $\Gamma$ \textit{traces} the edge $e$ of $\Gamma$ if 
the boundary of the ribbon that contains $e$ is a portion of $\delta$. There are at most two boundary components that can trace an edge, so there are at most four boundary components of a  ribbon graph that trace two edges from $\Gamma$. 

\begin{lemma}\label{lem:add-vertex}
Let $\Gamma,\Gamma'$ be two assembly graphs such that $\Gamma'=\Gamma(e_1,e_1',e_2,e_2')$ by crossing edges $e,e'$ and addition of  a new vertex $w$. Suppose $F$ and $F(w)$ are  ribbon graphs for $\Gamma,\Gamma'$ as defined above with $b,b'$ being the number of boundary components for $F$ and $F(w)$, respectively. Then we have the following:
\begin{itemize}
\setlength{\itemsep}{-3pt}
\item[\rm (i)] If both $e$ and $e'$ are traced by only  one boundary component in $F$,  then $b'=b+1$ or $b'=b+3$.
\item[\rm (ii)] If  both $e$ and $e'$ are traced by two boundary components  in $F$, then $b'=b+1$ or $b'=b-1$.
\item[\rm (iii)] If  $e$ and $e'$ are traced by three boundary components in $F$, then $b'=b-1$.
\item[\rm (iv)] If $e$ and $e'$  are traced by four boundary components in $F$, then  $b'=b-3$.
\end{itemize}
\end{lemma}

\begin{proof}
The proof follows directly from the observations shown in 
Figures  \ref{strand-connection} and  \ref{add-vertex-strands}.

(i) If both edges $e,e'$ are traced by a single component in $F$, then by crossing $e,e'$ the boundary components in $F(w)$ follow the situation (b), (f), and (j) in Figsure \ref{add-vertex-strands}. Situations (b) and (j) increase the number of components by 1 and (f) increase the number of components by 3. If the boundary connections at $w$ are
 changed (see Figure  \ref{strand-connection}) then in case (b) the number of components remains the same, in case (f) the number reduces from 4 to 2 and in case (j) the number increases from 2 to 4. Since the connections of the boundary components at all vertices in $F$ are the same as those in $F(w)$, the Lemma follows. 
 
 Cases (ii), (iii) and (iv) follow similarly.
 \end{proof}

%%%%%%%%%%%%%%%%%%%%%%
\subsection{Genus Range of the Tangled Cord $T_n$}\label{TCIntro}
%%%%%%%%%%%%%%%%%%%%%%

A {\it tangled cord} {\cite{BDJMS}}
with $n$ vertices and $2n$ edges, denoted $T_n$,  is a special type of assembly graph of the form illustrated in Figure  \ref{TC}.
The graph $T_n$ corresponds to the DOW $$ 1213243 \cdots (n-1)(n-2)n(n-1)n . $$
Specifically, $T_1, T_2$ and $T_3$correspond to $11$, $1212$, and $121323$ respectively,
  %%  $T_1=11$, $T_2=1212$, $T_3=121323$, 
and the DOW of %%
$T_{n+1}$ is obtained from the DOW of %%
  $T_{n}$ by replacing the last letter $(n)$ by
the subword $(n+1)n(n+1)$. Figure \ref{TC} shows the structure of the tangled cord.
For vertices and edges of $T_n$, we establish the following notation for the rest of the section.
%%The symbol $i$ of the DOW for $T_n$ corresponds to the vertex $v_i$.
The edges of $T_n$ are enumerated as $e_1, \ldots, e_{2n}$ as they are encountered by the transverse path given by the DOW, in this order. 
The adjacent edges to each vertex are listed below, in the cyclic order around the vertex as follows. 
\[\begin{tabular}{l}
$v_1:\ e_1,e_3, e_{2n}, e_2$ \\
$v_2:\ e_1,e_5, e_2, e_4$ \\
$v_i:\ e_{2(i-1)}, e_{2i}, e_{2(i-2)+1}, e_{2i+1}\mbox{ for }i\neq 1,2,n-1,n$\\
$v_{n-1}:\ e_{2(n-2)}, e_{2n-2}, e_{2(n-3)+1}, e_{2n-1}$\\
$v_n:\ e_{2n}, e_{2n-2}, e_{2n-1}, e_{2n-3}$
\end{tabular}\]
By construction, $T_{n+1}$ is obtained from $T_n$ by addition of a vertex $v_{n+1}$ by crossing $e_{2n-1}$ and $e_{2n}$ 
 as in Figure \ref{TC}, with the cyclic order  $e_{2n+1}, e_{2n-1},  e_{2n}, e_{2n-2}$ at $v_{n+1}$. 

\begin{figure}[h]
\begin{center}
\includegraphics[width=8cm]{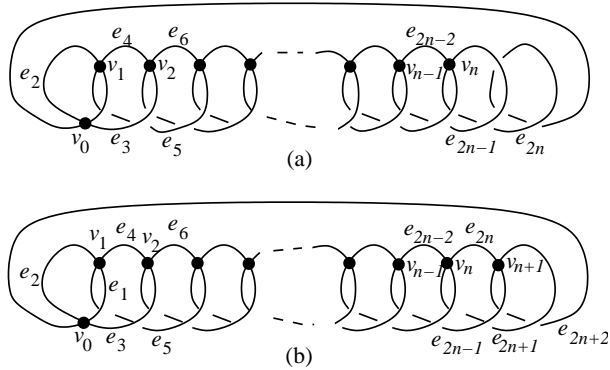} 
\caption{Tangled cord $T_n$}
\label{TC}
\end{center}
\end{figure}

The tangled cord $T_n$ was introduced in \cite{BDJMS} as a graph that provided a tight upper bound on the number of 
Hamiltonian polygonal paths (paths in which consecutive edges are neighbors with respect to their common incidence vertex visiting every vertex exactly once) over all assembly graphs of the same number of vertices.

We prove that  for each $T_n$, if $F$ is a  ribbon graph for $T_n$ then it has either 1 or 3 boundary components 
(if $n$ is odd) or it has 2 or 4 boundary components (if $n$ is even). 
First we prove the following lemma:
\begin{lemma}\label{TC-bound}
Let $F$ be a  ribbon graph for $T_n$ ($n\ge 2$) and suppose $B$ is the
 set of boundary components that trace edges $e_{2n-1}$ and $e_{2n}$. 
 Then every edge of $T_n$ is traced by a curve in $B$. 
\end{lemma}

\begin{proof} We use the following notation: for any set $E$ of edges, let $B(E)$ be the set of boundary components 
of the  ribbon graph $F$ that trace at least one edge in $E$.  So $B$ 
in the statement is written as $B(\{e_{2n-1},e_{2n}\})$. Observe that $B$ has at most four elements as there are at most four boundary components tracing two edges.
Now edges $e_{2n-1}$ and $e_{2n}$ in $T_n$  are not neighbors, but they are consecutive edges of the transverse path defining $T_n$. Hence all boundary components tracing $e_{2n-1}$ and $e_{2n}$ (which are at most four) must trace their neighboring edges $e_{2n-2}$ and $e_{2n-3}$ (see Figure  \ref{flip}(A) or (C)). Therefore 
$B(\{e_{2n},e_{2n-1},e_{2n-2},e_{2n-3}\})= B$. 
 But then edges $e_{2n-1}$ and $e_{2n-2}$ are consecutive edges of the transverse path at vertex $v_{n-1}$ and so, all boundary components that trace $e_{2n-1}$ and $e_{2n-2}$ must also trace their neighboring edges at 
 $v_{n-1}$, that is,  edges $e_{2n-4}$ and $e_{2n-5}$. So $B(\{e_{2n},\ldots, e_{2n-5}\})=B$.
 
 Inductively, if $B(\{e_{2n},\ldots,e_{2i},e_{2i-1}\})=B$, then all edges incident to vertices $v_n, v_{n-1},\ldots,v_{i+1}$
 ($i\ge 2$) are traced with boundary components from $B$. The edges $e_{2i+1}$ and $e_{2i}$ are not neighbors
 and are both 
  incident to $v_i$, therefore the boundary components that trace $e_{2i+1}$ and $e_{2i}$, also trace 
 $e_{2i-2}$ and $e_{2i-3}$. Hence $B(\{e_{2n},e_{2n-1},\ldots, e_{2i-2},e_{2i-3}\})=B$. If $i=2$ we obtain that 
 $B(E)=B$ where $E$ is the set of all edges of $T_n$.
 \end{proof}
 
 As a consequence of Lemma  \ref{TC-bound}, and the fact that the parity of the number of boundary components of the  ribbon graph must match the parity of the vertices from Lemma~\ref{Eulerlem}, 
  we have that if $n$ is even then any  ribbon graph of $T_n$ has 2 or 4
 boundary components, and if $n$ is odd, then any  ribbon graph of $T_n$ has 1 or 3 boundary components. In the following we observe that all these situations appear.

  \begin{figure}[h]
\begin{center}
\includegraphics[width=15cm]{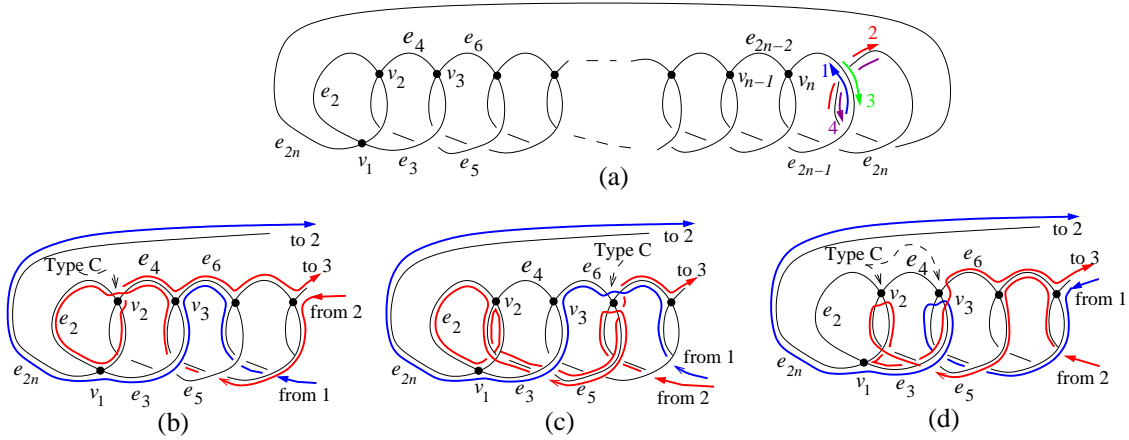}
\caption{Connections of four boundary components in $T_n$.}
\label{TC-4strands}
\end{center}
\end{figure}

 \begin{lemma}\label{TC-three}
  For every odd number $n>2$ there is a  ribbon graph of $T_n$ with three  boundary components.
 \end{lemma}
 
\begin{proof} %%% the proof is redundant so one sentence is sufficient
This lemma is a corollary of (and the proof of) Lemma~\ref{fullrangelem} where it is observed that for every assembly graph there is a ribbon graph with more than one boundary component. 
\end{proof}

%  The four possible boundary components tracing edges $e_{2n}$ and $e_{2n-1}$ are indicated in Figure  
%\ref{TC-4strands}, 
% also numbered 1 through 4.
%Consider the  ribbon graph $F$ constructed from 
% a fixed planar diagram  in Figure  \ref{TC-4strands}(a), 
% where the connections of boundary curves 
% at vertices $v_1,\ldots, v_{n-1}$ are as in Figure \ref{flip}(A), called type (A) 
 % connection, except at vertex $v_n$ where the connection is as in Figure  \ref{flip}(C), called type (C). 
 % If one follows the boundary curve labeled $4$ of edge $e_{2n}$, this curve traces
 %  $e_{2n}$ and  visits vertices $v_1, \ldots, v_{n-1}$ by tracing  even numbered edges.
 %  Then  at $v_n$,  it follows the type (C) connection and 
% traces $e_{2n}$, and goes back to itself. Hence every  ribbon graph with these connections
 % must have at least two boundary components, because the odd numbered edges are not traced by this boundary 
 % curve. Therefore, for all odd numbers $n\ge 3$, there is a
  %  ribbon graph for $T_n$ with three boundary components.
%\end{proof}

\begin{lemma}\label{TC-four}
For every odd number $n>2$,  there is a   ribbon graph $F$ with a single boundary component tracing edges $e_{2n}$
and $e_{2n-1}$  whose global connection is as in Figure \ref{add-vertex-strands}(f).
\end{lemma}

\begin{proof}
Since $n$ is odd, the number of boundary components are either 1 or 3. 
The global connection of the boundary component in Figure \ref{add-vertex-strands}(f) implies that the four 
possible boundary components tracing edges $e_{2n}$ and $e_{2n-1}$ (indicated in 
Figure  \ref{TC-4strands}(a)
with 1 through 4)  are connected such that 1 is connected to 2, 2 connected to 3.
If $3$ is connected to $1$, there would be two boundary components, hence 3 must be connected to 4, and back to 1. 
Therefore we only need to show that there is a  ribbon graph where 1 is connected to 2  and 2 to 3. 
Suppose a  ribbon graph 
corresponding to a fixed planar diagram  in Figure \ref{TC-4strands}(a)
has boundary curves connected as type (A) of Figure \ref{flip}
at vertices $v_n, v_{n-1}, \ldots, v_5$ and also at vertex $v_1$. Then  following the indicated arrows 
in Figure  \ref{TC-4strands}(a)  the boundary curves visit vertices:
$$
\begin{array}{lllllll}
v_n & v_{n-2} & v_{n-3} & v_{n-5} & v_{n-6} & v_{n-8}  & v_{n-9}  \ldots 
\mbox{by following boundary curve 1}\qquad(*)\\
v_n & v_{n-1} & v_{n-3} & v_{n-4} & v_{n-6} & v_{n-7} & v_{n-9}  \ldots 
\mbox{by following boundary curve 2}\qquad(**)
\end{array}
$$

It follows that the two curves  ``meet" at every other vertex they visit 
(with indices changing in gaps of 3), but if the two curves ``meet'' at $v_i$,  one of the curves  traces edge $e_{2i}$
and the other traces $e_{2i+1}$. There are three possibilities for the vertex of the 
smallest index  (greater than 2) for $(*)$ and $(**)$ to ``meet":
 at vertex $v_3$, vertex $v_4$ and at vertex $v_5$ respectively. 
 These situations are depicted in Figure  \ref{TC-4strands}(b--d, respectively) where the blue 
 component indicates (*) and the red indicates (**). 
 
The curves ``meet" at vertex $v_3$ when  $n=6k+3$ 
 (Figure  \ref{TC-4strands}(b)), in which case we consider the ribbon graph where the boundary curves 
 are connected 
  at vertex $v_2$ in  type (C) in Figure \ref{flip}, and at all other vertices the connections are of type (A). 
The curves ``meet" at vertex $v_4$ when $n=6k+1$
 (Figure  \ref{TC-4strands}(b)), in which case we consider the ribbon graph where the boundary curves 
 are connected 
  at vertex $v_4$ in  type (C), and at all other vertices the connections are of type (A).   
  The curves ``meet" at vertex $v_5$ when $n=6k+5$
 (Figure  \ref{TC-4strands}(b)), in which case we consider the ribbon graph where the boundary curves 
 are connected 
  at vertex $v_2$ and $v_3$ in  type (C), and at all other vertices the connections are of type (A).
  In all three cases we note that: (i)  the (blue) boundary curve  1  that traces $e_3$ to $v_1$ 
and then $e_{2n}$, continues to ``join" with boundary curve 2 and (ii) the (red)
 boundary curve 2 that traces the even edges 
$e_6, e_8, e_{10}, \ldots,  e_{2n-2}$, and then $e_{2n-1}$ ``joins" with boundary curve 3 (Figure  \ref{TC-4strands}(a)). Since in all three cases  situations (i) and (ii)  appear, the
  global connection of the boundary components at edges $e_{2n-1}$ and $e_{2n}$ are the same as in 
   Figure \ref{add-vertex-strands}(f).
\end{proof}

 The following theorem gives the final result about the genus range of $T_n$.
 
\begin{theorem}\label{TCthm} Let $T_n$ be the tangled cord with $n$ vertices.
Then 
\[\ 
\gr(T_n)=   \begin{cases}
\left\{  \,  \frac{n-2}{2}, \,  \frac{n}{2}\, \right\}  & \text{if } n \text{ is even,}\\
\left\{  \,   \frac{n-1}{2}, \,  \frac{n+1}{2} \, \right\} & \text{if } n \text{ is odd.}
  \end{cases}
\]
\end{theorem}
\begin{proof}
Recall that there are at most 4 boundary components by Lemma \ref{TC-bound}. 
By Lemmas  \ref{TC-three}, \ref{TC-four}  for every odd $n$, there is a  ribbon graph for $T_n$ with one boundary component and there is a  ribbon graph for $T_n$ with three boundary components. We  need to prove,
for $n$ even, that there are ribbon graphs with  2 and 4 boundary components. 
 The graph $T_{n+1}$ is obtained from $T_n$ by vertex addition by crossing $e_{2n-1}$ and $e_{2n}$. 
 Suppose $n\ge 3$ is odd. Then $T_n$ has  ribbon graphs with one and three components. 
By Lemma \ref{lem:add-vertex}(iii) by adding vertex $v_{n+1}$ through crossing $e_{2n-1}$ and 
$e_{2n}$ the three-component  ribbon graph for $T_n$ becomes a 
ribbon graph with two boundary components 
for $T_{n+1}$.
By Lemma \ref{TC-four} there is a   ribbon graph for $T_n$ with a single boundary component, in which the global boundary connection traces edges $e_{2n-1}$ and $e_{2n}$ as  in Figure  \ref{add-vertex-strands}(f). 
Then after  adding vertex $v_{n+1}$ by crossing $e_{2n-1}$ and 
$e_{2n}$, this  ribbon graph becomes a ribbon graph
with four boundary components 
for $T_{n+1}$. By Lemma \ref{Eulerlem}  the result follows.
\end{proof}

\begin{corollary}
The maximum genus range of assembly graphs with $2n-1$ vertices is $[n-1, n]$ for any $n \in \N$.
\end{corollary}
\begin{proof}
This follows from Lemma~\ref{notnnlem} and Theorem~\ref{TCthm}.
\end{proof}

\begin{conjecture}
{\rm 
The maximum genus range of assembly graphs with $2n$ vertices is $[n-1, n]$ for any $n \in \N$.
}
\end{conjecture}

This conjecture follows from another conjecture:  $[n, n] \notin \GR_{2n}$ for any $n \in \N$,
together with 
Corollary \ref{cplem}  and Theorem~\ref{TCthm}.

%%%%%%%%%%%%%%%
\section{Towards Characterizing Genus Ranges}\label{Characterize}
%%%%%%%%%%%%%

In this section we summarize our results and use inductive arguments to determine 
genus rages we can characterize at this point. 
To describe the 
genus  %%
  ranges we can realize, we define the following integer sequence.
For positive integers $k, \ell \in \Z_{\geq 0}$, 
let $\phi ( k, \ell )  = 7 k + 3 \ell - 1$.
 For any $n \in \N$, define 
$K_n= {\rm max} \{ k \in \Z_{\geq 0} \, |  \,  \phi ( k, \ell )  \leq n \}$, $L_n= {\rm max} \{ \ell \in \Z_{\geq 0} \, |  \,  \phi ( K_n, \ell )  \leq n \}$ and 
$\psi_n= 3 K_n + L_n $.
For $n=100$, for example, $K_n=14$, $L_n=1$ and $\Psi_n=43$. 
For small integers, one computes their values as follows.

\bigskip

\noindent
\begin{tabular}{l||rrrrrrrrrrrrrrrrrrrrrrrrrr}
\hline
n             & 1&  2&  3&  4&  5&  6&  7&  8&  9& 10& 11& 12& 13& 14& 15& 16& 17& 18& 19& 20& 21&  22 \\  \hline  \hline
$K_n $  & 0&  0&  0&  0&  0&  1&  1&  1&  1&    1&  1&   1&    2&    2&   2&   2&    2&   2&    2&   3&   3&  3\\ \hline
$L_n $   & 0& 1&  1&  1&  2&  0&  0&  0&  1&    1&  1&    2&    0&    0&   0&  1&    1&   1&     2&   0&  0&   0  \\ \hline
$\psi_n$& 0& 1&  1&  1&  2&  3&  3&  3&  4&    4&   4&   5&   6&    6&    6&  7&    7&   7&     8&   9&  9&   9 \\ \hline
\end{tabular}

\bigskip

Recall from Corollary~\ref{cplem} that for $n \in \N$,  $ \GR_{2n-1}$  and $\GR_{2n}$
are subsets of  ${\cal CP}(n)$. Hence towards Problem~\ref{MainProblem}, we consider
 which sets in ${\cal CP}(n)$ are realized as elements of  $ \GR_{2n-1}$  and $\GR_{2n}$,
and  which sets are proved to be not realized. 

\begin{theorem}\label{mainthm} {\ }  %% move (i) in next line

\noindent
{\rm (i)} For any $n\in \N$, 
$ \GR_{2n-1}$ contains the set  % \supset  \left[\, 
 ${\cal CP}(n) \setminus \{ [0, n], [h,h] \, | \,  \psi_{2n-1} < h \leq n\}$.  % \, \right] $.

\noindent
{\rm (ii)} For  any $n\in \N$, 
$ \GR_{2n} $ contains the set  %\supset  \left[\,   
${\cal CP}(n)\setminus \{ [h,h]\, | \, \psi_{2n} < h \leq n \}$.  % \, \right]  $.

\noindent
{\rm (iii)} For   any $n\in \N$,  
$ [0, n],  [n,n] \notin \GR_{2n-1}$. 
\end{theorem}

\begin{proof}
The  part (iii) is a restatement of Lemmas~\ref{fullrangelem} and \ref{notnnlem}. We show (i) and (ii) by induction and constructions.
Computer calculations show that the statements hold for $n=1,2,3$ for $\GR_{2n-1}$ and $\GR_{2n}$. 

First we focus on sets $[a,b]$ with $a < b$. 
Assume   for induction that $\GR_{2k-1}$ contains 
all sets $[a, b] \in {\cal CP}(k)$, where $a <b \leq k$, excluding $[0,k]$.
By Lemma~\ref{higherlem},  
all these sets are   contained in  $\GR_{2k}$. 
Proposition~\ref{halfgrprop} implies that  $[0, k]$ is contained in $\GR_{2k}$ and so is in $\GR_{2k+1}$. 
Hence all sets $[a, b] \in {\cal CP}(k)$, where $a <b \leq k$ for $n \leq k$, are contained in  $\GR_{2k}$
and    $\GR_{2k+1}$.

Assume next  that  $\GR_{2k}$ contains all sets $[a, b] \in {\cal CP}(k)$, where $a <b \leq k$, 
and also that $\GR_{2k-1}$ contains all sets $[a, b] \in {\cal CP}(k)$, where $a <b \leq k$, excluding $[0,k]$.
By Lemma~\ref{higherlem},  all $[a, b] \in {\cal CP}(k)$, where $a <b \leq k$, are contained in $\GR_{2k+1}$. 
We show that $\GR_{2k+1}$ contains  all $[a, b] \in {\cal CP}(k+1)$, where $a <b \leq k+1$, excluding $[0,k+1]$.
Among these wanted, the following have not been covered by the induction assumption:
$ [h, k+1] $ for $h=1, \ldots, k+1$ and $[0, k]$. 
There is an assembly graph with $2k+1$ vertices and genus range $[0, k]$
by Proposition~\ref{oddhalfgrprop}. 
Let $\Gamma$ be an assembly graph with $2k-1$ vertices that have  genus range $[h, k]$ for a fixed $h$, 
where $1 \leq h \leq k-1$.
The ribbon graph of $\Gamma$ with the highest genus $k$ has connected boundary component.  
Therefore the condition (ii) of Lemma~\ref{plus1lem} is satisfied with any edge of $\Gamma$.
Hence by Lemma~\ref{plus1lem}, there is an assembly graph with $2k+1$ vertices and with the genus range $[h, k+1]$. 
Thus $\GR_{2k+1}$ contains  all $[a, b] \in {\cal CP}(k+1)$, where $a <b \leq k+1$, excluding $[0,k+1]$.

Finally we show that $\GR_{n}$ contains $[h,h]$ for $ 0 \leq h \leq \psi_{n} $.
Let  $\hat \Gamma$  %$\Psi$ 
 be the graph with $6$ vertices and genus range $\{ 3 \}$ corresponding to the word $123245153646$
mentioned in Remark~\ref{highest-single-rem}, Item 5. 
Assume that $\GR_{n}$ contains $[h,h]$ for $ 0 \leq h \leq \psi_{n} $ for $n \leq k$. 
For $n=k+1$, let $m=K_n+L_n$.  %% 
Consider the assembly graph $\Gamma_n$ obtained by cross sum construction %be the assembly graph 
  as depicted in Figure~\ref{connectall},  %\ref{repeat}, 
  where
in the boxes $B_1, \ldots, B_{K_n}$ copies of the graph $\hat \Gamma$  %$\Psi$ 
  are inserted, while  %%%
   in $B_{K_n +1}, \ldots, B_{K_n + L_n}$ copies of 
the graph
$\Gamma_0$ corresponding to the word $1212$ are inserted.  %, with $m=K_n+L_n$. 
By Lemma~\ref{joinlem}, we have $\gr(\Gamma_n)=[\psi_n, \psi_n]$, and it holds for $n=k+1$.
\end{proof}

\begin{remark}
{\rm 
{}From Theorem~\ref{mainthm}, we conjecture  that for any $m \in \N$, %change n to m 
there is an integer $\Psi_m\ge \psi_m$ such that %% added
$$\begin{array}{lll}
\GR_{2n} &=& {\cal CP}(n) \setminus \{ \, [h,h]\, | \, \Psi_{2n} < h \leq n \, \} , \\
\GR_{2n-1} &=& {\cal CP}(n) \setminus \{ \, [0, n], [h,h]\, | \, \Psi_{2n-1} < h \leq n \, \}.  % ,
\end{array} $$
%%   for some integer sequence $\Psi_{m}$ such that $\psi_m \leq \Psi_m$.

For small values of $m\leq 20$, the conjecture holds with $\Psi_m=\psi_m$  for 
$m=1, \ldots, 7, 9, 13$. At this time we are not able to determine if 
$[5,5]$ is in $\GR_m$ for $m=10, 11$, $[6,6]$ for $m=12$,  $[7,7]$ for $m=14,15$, 
$[8,8]$ for $m=16, 17, 18$, and $[9,9]$ for $m=18, 19$.
The conjecture is true if there is only one unknown entry, and therefore, 
the conjecture holds for all $m\leq 20$  except $m=18$, and $m=18$ is the smallest number for which 
we do not know if the conjecture holds. If one finds $[8,8] \notin \GR_{18}$ but $[9,9]\in  \GR_{18}$,
then this will provide a counterexample. 
For $n=100$, for example, we do not know if $[h,h] \in \GR_{100}$ for 
$h=43,  \ldots, 50$. 
We know that $[4,4]$ is not in $\GR_8$ only by % from 
computer calculations
searching through all assembly graphs with 8 vertices.  %% added
}
\end{remark}

%here

%%%%%%%%%%%%%%
\section{Concluding Remarks}\label{ConclSec}
%%%%%%%%%%%%%%

For regular 4-valent rigid vertex graphs, we defined the genus range, that are genera of cellular embeddings.
Some ranges are realized by using specific families of graphs, and by some operations applied on graphs of smaller sizes. 
A few ranges are shown to be not realized as genus range. 
We identified certain ranges of singletons $[h,h]$ that remain unknown  whether they can be genus ranges. 

In Figure \ref{fig:genus_ranges_7_8}(right), we note that there are only 13 graphs  among 65346 graphs 
of 8 vertices with genus range $[0, 4]$. In Theorem~\ref{halfgrprop},
we constructed graphs with $2n$ vertices with genus range $[0,n]$, and these graphs seem to be very rare.
Also we notice that among 7 vertex graphs, there are only 2 graphs  with genus range $[3,3]$ (see Remark \ref{highest-single-rem}, Item 5 and 6). 
It is a curious fact that certain types of genus rages are very rare, and some are numerous.

We remark on possible relations and applications to the studies on DNA assembly.
A mathematical model using assembly graphs for DNA recombination processes is proposed and studied in
\cite{Angeleska2009,Angeleska2007,BDJMS}, for example. 
The molecular structure in space at the exact moment of recombination is modeled by assembly graphs, and the assembled gene is modeled by {\it Hamiltonian polygonal paths}, that are paths that make  %s 
``$90^\circ$ turn"  %% added quotes
   at every rigid vertex, and visit every vertex exactly once.
Polygonal paths  %Such paths 
  have been also studied in graph theory as A-trails (for example, \cite{ABJ}).
In \cite{BDJMS}, it was proved that the tangled cord,
which achieves the maximal genus range for at least the graphs with odd number of vertices,  %% added 
has the largest number of Hamiltonian polygonal paths among all
assembly graphs of the same number of vertices.
 If one follows a boundary curve of a  ribbon graph %of 
 at every vertex (see Figure  \ref{flip}), 
  the consecutive edges traced by the boundary curve must be neighbors, so they form a polygonal path. 
Thus it is expected that smaller numbers of boundary components   contribute to larger numbers of Hamiltonian 
polygonal paths. 
However, at this time we don't know the nature of the relationship between the genus range and the number of 
Hamiltonian polygonal paths.
%%% More studies on this relation are desirable. 

\section*{Acknowledgements}
We wish to thank F. Din-Houn Lau and Kylash Rajendran, Erica Flapan, Mauro Mauricio and Julian Gibbons for insightful discussions.   DB is supported in part by EPSRC Grants EP/H0313671, EP/G0395851 and EP/J1075308, and thanks the LMS for their Scheme 2 Grant.  KV is supported by EP/G0395851. The research was partially supported by the National Science Foundation under Grants No. DMS-0900671 and CCF-1117254.

%%%%%%%%%%%%%%%%

\end{document}